\documentclass[10pt,a4paper,bibliography=totoc,oneside,reqno]{amsart}

\usepackage[utf8]{inputenc}
\usepackage[english]{babel}

\usepackage{amsmath,amssymb,amsthm}
\usepackage{bbm}
\usepackage{tikz-cd}
\usepackage{mathdots}
\usepackage{cancel}

\usepackage[totalwidth=14cm,totalheight=21.5cm]{geometry}
\usepackage{hyperref}
\usepackage{lipsum,xpatch}

\usepackage[backend=bibtex,
style=alphabetic,
sorting=nyt,
backref,
backrefstyle=none,
maxnames=4,
maxbibnames=999]{biblatex}
\theoremstyle{plain}
\newtheorem{thm}{Theorem}[section]
\newtheorem*{thm*}{Theorem}
\newtheorem{pro}[thm]{Proposition}
\newtheorem{cor}[thm]{Corollary}
\newtheorem{lem}[thm]{Lemma}

\makeatletter
\newtheorem*{rep@theorem}{\rep@title}
\newcommand{\newreptheorem}[2]{%
	\newenvironment{rep#1}[1]{%
		\def\rep@title{#2 \ref{##1}}%
		\begin{rep@theorem}}%
		{\end{rep@theorem}}}
\makeatother

\newtheorem{thmx}{Theorem}
\newreptheorem{thmx}{Theorem}

\newcommand{\Scal}{\mathrm{Scal}}

\newtheorem{prox}[thmx]{Proposition}
\newreptheorem{prox}{Proposition}

\newreptheorem{lemx}{Lemma}

\newreptheorem{corx}{Corollary}

\newtheorem{proxx}{Theorem}
\newreptheorem{proxx}{Theorem}

\theoremstyle{remark}
\newtheorem{rem}[thm]{Remark}
\newtheorem{de}[thm]{Definition}

\makeatletter
\newcounter{step}
\xpretocmd{\proof}{\setcounter{step}{0}}{}{}
\newcommand{\step}[1]{%
	\par
	\addvspace{\medskipamount}%
	\stepcounter{step}%
	\noindent\emph{Step \thestep: #1.}\par\nobreak\smallskip
	\@afterheading
}
\makeatother

\newcommand{\R}{\mathbb{R}}
\newcommand{\N}{\mathbb{N}}
\newcommand{\hyp}{\mathbb{H}}
\newcommand{\hpq}{\hyp^{p,q}}
\newcommand{\sph}{\mathbb{S}}
\newcommand{\proj}{\mathbb{P}}

\newcommand{\sff}{\mathrm{I\!I}}
\newcommand{\A}{\mathrm{A}}
\newcommand{\AZ}{\mathsf{A}_0}
\newcommand{\Q}{\mathsf{Q}}
\newcommand{\Einpq}{\mathsf{Ein}^{p,q}}

\renewcommand{\P}{\mathcal{P}}
\newcommand{\T}{\mathsf{T}}
\newcommand{\No}{\mathsf{N}}
\newcommand{\tr}{\mathrm{trace}}
\newcommand{\V}{\mathbf{V}}
\newcommand{\ATN}[1]{\mathcal{A}^1({#1},\mathrm{Hom}(\T {#1},\No {#1}))}
\newcommand{\OTN}[1]{\Omega^1({#1},\mathrm{Hom}(\T {#1},\No {#1}))}
\newcommand{\scal}[2]{\langle #1,#2\rangle}
\newcommand{\X}{\mathcal{X}}
\newcommand{\nb}{{\bar{\nabla}}}
\newcommand{\dn}{{\mathrm{d}_\nb}^{} }
\newcommand{\dns}{{\mathrm{d}_\nb}^{*}}

\renewcommand{\leq}{\leqslant}
\renewcommand{\le}{\leqslant}
\renewcommand{\ge}{\geqslant}
\renewcommand{\geq}{\geqslant}

\DeclareMathOperator{\ric}{Ric}

\title[Scalar curvature of maximal submanifolds in $\hyp^{p,q}$]{On the scalar curvature of complete maximal spacelike submanifolds in pseudo-hyperbolic spaces}
\author{Alex Moriani}
\address{Laboratoire J.A. Dieudonn\'e\\
Université Côte d'Azur\\
France}
\email{amoriani@unice.fr}
\author{Enrico Trebeschi}
\address{Laboratoire J.A. Dieudonn\'e\\
Université Côte d'Azur\\
France}
\email{enrico.trebeschi@univ-cotedazur.fr}
\date{}

\thanks{A.~Moriani and E.~Trebeschi acknowledge funding by the European Research Council under ERC-Advanced grant 101095722 AnSur (Geometric Analysis and Surface Groups), ERC-consolidator grant 101124349 GENERATE (GeomEtry and aNalysis for $(G, X)-$structurEs and their
defoRmATion spacEs). Views and opinions expressed are however those of the author(s) only and do not necessarily reflect those of the European Union or the European Research Council Executive Agency. Neither the European Union nor the granting authority can be held responsible for them. This collaboration have been partially funded by ANR JCJC grant GAPR (ANR-22-CE40-0001, Geometry and Analysis in the Pseudo-Riemannian setting), and PRIN F53D23002800001.}

\begin{document}
\begin{abstract}
    We study in this article the curvature of complete maximal spacelike submanifolds of in pseudo-hyperbolic spaces. We show that the scalar curvature of these submanifolds is nonpositive. This gives, together with a result of Ishihara, a sharp bound on the scalar curvature of complete maximal $p$-submanifolds of $\hpq$. We show that achieving such bound at a point is equivalent to achieving it identically, and we explicitly describe the submanifolds achieving the bound. When $q=1$, we deduce a sharp upper bound on the Ricci curvature of complete maximal hypersurfaces in Anti-de Sitter space, and characterize the hypersurfaces achieving it.
    
    Finally, we discuss the link between scalar curvature and Gromov-hyperbolicity for complete maximal spacelike $p-$submanifolds in $\mathbb{H}^{p,q}$.\end{abstract}
\maketitle

\tableofcontents

\section{Introduction}
The \textit{pseudo-hyperbolic space} is a geodesically complete pseudo-Riemanninan manifold of sectional curvature $-1$. In this article, we consider the \textit{quadric model} for the pseudo-hyperbolic space, denoted by $\hyp^{p,q}$, namely the space
\[\hpq:=\{x\in\V\ |\ \Q(x)=-1\}\ ,\]
where $\V$ is a vector space endowed with a non-degenerate quadratic form $\Q$ of signature $(p,q+1)$. The space $\hpq$ is equipped with the metric induced by restricting $\Q$ to $\T\hyp^{p,q}$, whose fiber over a point $x$ of $\hpq$ equals $x^{\perp_\Q}$.

A submanifold $M$ is \textit{spacelike} if the induced metric is Riemannian, and $M$ has \textit{parallel mean curvature} (PMC) if its mean curvature $H$ is parallel with respect to the connection on the normal bundle of $M$. In particular, $M$ is \emph{maximal} if $H$ identically vanishes.

The study of PMC submanifolds is a classic problem in differential geometry, since they are solution of a nice variational problem, involving an area-like operator. More recently, it has captured the interest of geometric topologists in the contest of \textit{higher higher Teichm\"uller theory} and the theory of \textit{Anosov representations}, as we will explain in the following.

\subsection{Main results} Our first result consists in giving a sharp upper bound for scalar curvature of complete maximal spacelike $p-$submanifolds in $\hyp^{p,q}$, together with a complete classification of the $p-$submanifolds achieving the bound.

The submanifolds achieving the bound are \textit{product submanifolds} (see Definition~\ref{de:barbot}), which are a particular class of spacelike complete PMC $p-$submanifolds of $\hyp^{p,q}$: they are embeddings of $c_1\hyp^{n_1}\times\dots\times c_k\hyp^{n_k}$ in $\hyp^{p,q}$ for $k\le q+1$, $c_i>0$ and $n_1+\dots+n_k=p$.

Among them, \textit{pseudo-flats} (Definition~\ref{de:pseudoflat}) play an important role: they are are product submanifolds characterized by the fact that all $n_i$ equal 1. The name comes from the fact that they are orbits by the action of the Cartan subgroups of $\mathrm{Isom}\hyp^{p,q}$, as shown in Subsection~\ref{sub:cartan}), hence they generalize maximal flats in symmetric spaces. Moreover, they are isometric copies of $\R^p$ inside $\hyp^{p,q}$, hence flat.

\begin{proxx}\label{pro:scal}
Let $M$ be a complete maximal spacelike $p-$submanifold in $\hyp^{p,q}$, then \[\Scal_M\leq 0.\] Moreover, if the bound is achieved at one point, then $M$ is a pseudo-flat.
\end{proxx}

By tracing twice \ref{Gauss equation}, one recovers an important equation relating the norm of the second fundamental form with the scalar curvature:
\begin{equation}\label{eq:trace}
    \Scal_M=-p(p-1)-\|H\|^2+\|\sff_M\|^2,
\end{equation}
for $H$ the mean curvature of $M$. This allows to give an interpretation of Theorem~\ref{pro:scal} in terms of the extrinsic geometry of maximal $p-$submanifolds.

\begin{proxx}\label{pro:II}
Let $M$ be a complete maximal spacelike $p-$submanifold in $\hyp^{p,q}$, then \[\|\sff_M\|^2\le p(p-1).\] Moreover, if the bound is achieved at one point, then $M$ is a pseudo-flat.
\end{proxx}

An interesting consequence of Theorem~\ref{pro:scal} and Theorem~\ref{pro:II} is that both the intrinsic and extrinsic geometry of complete maximal spacelike $p-$submanifolds are uniformly bounded by quantities depending only on the dimension $p$ and not on the timelike codimension $q$. In a broad sense, this result confirms the intuition behind \cite[Theorem~6.5]{lt23} and \cite[Theorem~1.1]{mazviag}, where the intrinsic and the extrinsic distance on maximal $p-$submanifolds differ by constants only depending on the spacelike dimension $p$.

Moreover, our result proves that the $p-$submanifolds achieving the bound are contained in a totally geodesic copy of $\hyp^{p,p-1}$ inside $\hyp^{p,q}$.

We recall that Ishihara provides a different bound for the second fundamental form of complete spacelike maximal $p-$submanifolds in $\hyp^{p,q}$:
\begin{thm*}[{\cite[Theorem 1.2 and 1.3]{ish88}}]
    Let $M$ be a complete maximal spacelike $p-$submanifold in $\hyp^{p,q}$, then \[\|\sff_M\|^2\le pq.\] Moreover, if the bound is achieved at one point, then $M$ is a product submanifold with $q+1$ factors.
\end{thm*}

We will discuss in detail in Subsection~\ref{sub:ingredients} the difference between the two approaches and results. For the moment, by combining our result with Ishihara's one, we are able to give a sharp upper bound on the scalar curvature, or equivalently the norm of the second fundamental form, of maximal $p-$submanifolds in $\hyp^{p,q}$:

\begin{thmx}\label{thm:scal}
Let $M$ be a complete maximal spacelike $p-$submanifold in $\hyp^{p,q}$, then \[\Scal_M\leq p\cdot\min\{0,q-p+1\}.\] Moreover, if the bound is achieved at one point, then $M$ is a product submanifold of the form $\hyp^{n_1}\times\dots\times\hyp^{n_k},$ for $k=\min\{p,q+1\}$.
\end{thmx}

\begin{rem}
    Incidentally, this result proves that there are no flat nor Ricci-flat complete maximal spacelike $p-$submanifolds in $\hyp^{p,q}$, for $q<p-1$.
\end{rem}

\begin{thmx}\label{thm:II}
Let $M$ be a complete maximal spacelike $p-$submanifold in $\hyp^{p,q}$, then \[\|\sff_M\|^2\le p\cdot\min\{p-1,q\}.\] Moreover, if the bound is achieved over $M$, then $M$ is a product submanifold of the form $\hyp^{n_1}\times\dots\times\hyp^{n_k},$ for $k=\min\{p,q+1\}$.
\end{thmx}

\begin{rem}
    We will carefully describe spacelike product submanifolds in Section~\ref{sec:product}. For the moment, we anticipate that spacelike product $p-$submanifolds are characterized by having parallel second fundamental form and flat normal bundle (Proposition~\ref{pro:parallel}). As a consequence, all geometric quantities depending on the second fundamental form, such as its scalar, Ricci and sectional curvature are parallel over spacelike product submanifolds.
\end{rem} 

\vspace{0.3cm}\noindent\textit{Anti-de Sitter space.} In the Lorentzian setting, that is for $q=1$, we are able to promote the upper bound on the scalar curvature (Theorem~\ref{thm:scal}) to an upper bound on the Ricci curvature. The maximal hypersurface achieving the bound on the Ricci curvature is unique, up to an ambient isometry.
\begin{thmx}\label{thm:Ricci}
Let $M$ be a properly embedded spacelike maximal hypersurface in the Anti-de Sitter space $\hyp^{p,1}$, then $M$ is Ricci non-positive.

Moreover, if $\ric$ vanishes at a vector $v$ of $\T M$, then $M$ is a product hypersurface of the form $\hyp^{p-1}\times\hyp^1$.
\end{thmx}
This result enlights a difference between maximal product hypersurfaces: in $\hyp^{p,1}$, any maximal product hypersurface $\hyp^{p-k}\times\hyp^{k}$ maximizes the scalar curvature, as stated in Theorem~\ref{thm:scal},  but the Ricci curvature is strictly negative if and only if $k\ne1,p-1$.

The discrepancy between product hypersurfaces in Anti-de Sitter space had already been noticed in the work of \cite{kkn91}: the sharp bound for the norm of the second fundamental form of CMC hypersurfaces in $\hyp^{p,1}$ is in fact achieved only by product hypersurfaces of the form $\hyp^{p-1}\times\hyp^1$, when $H\ne0$. 

Theorem~\ref{thm:Ricci} proves that such a difference stands even for maximal hypersurfaces, and it is detected by Ricci curvature, while scalar curvature is a too weak invariant to distinguish them for $H=0$.

\begin{rem}
    It could seem that we weakened the hypothesis on the maximal hypersurfaces, since Theorem~\ref{thm:scal} requires maximal hypersurfaces to be \textit{complete}, while in Theorem~\ref{thm:Ricci} we only ask them to be \textit{properly embedded}. However, in the Anti-de Sitter space, the two conditions are equivalent in the class of CMC hypersurfaces (\cite[Theorem~B]{ecrin}).
\end{rem}

\subsection{Historical background}
Quantitative estimates and rigidity results for complete spacelike PMC $p-$submanifolds in pseudo-Riemannian \textit{spaceforms}, that is geodesically complete pseudo-Riemannian manifolds with constant sectional curvature, are a classical problem in differential geometry.

\vspace{0.3cm}\noindent\textit{Extrinsic curvature.} After the pioneering works \cite{law,ckd} on the norm of the second fundamental form of minimal hypersurfaces in the Riemannian sphere $\sph^n$, several achievements had been obtained for complete spacelike PMC $p-$submanifolds in pseudo-Riemannian spaceforms, and it would not be possible for us to mention each one of them. 

To our knowledge, the first contribution concerning the pseudo-hyperbolic space is \cite{ish88}. In that work, the norm of the second fundamental form of a spacelike complete maximal $p-$submanifold in $\hyp^{p,q}$ is proved to be at most $pq$. Moreover, if the bound is achieved at one point, then the second fundamental form is parallel and the normal bundle is flat. Hence, Theorem~\ref{pro:II} improves Ishihara's bound for $q\ge p$, for any value of $p\ge2$.

For $p=2$ and $H=0$, the bound $2q$ of \cite{ish88} had already been sharpened independently in \cite{che94,lt23}, to
\[2=2\cdot\min\{1,q\}=p\cdot\min\{p-1,q\}.\] Hence, for the case $p=2$, Theorem~\ref{pro:II} retrieves their result, but following a different approach to the problem.

The results contained in \cite{ish88} have been generalized in \cite{kkn91} to complete spacelike CMC hypersurfaces in the Anti-de Sitter space $\hyp^{p,1}$, and in \cite{cc93} to complete spacelike PMC $p-$submanifolds of $\hyp^{p,q}$, for any value of $q$. In these works, the upper bound for the norm of the second fundamental form of a spacelike complete PMC $p-$submanifold $M$ in $\hyp^{p,q}$ is proven to be an explicit smooth function $S(p,q,\|H\|)$, for $H$ the mean curvature of $M$, which is strictly increasing in each variable. 

Since such function satifies $S(p,q,0)=pq$, combining Theorem~\ref{thm:II} together with a compactness result for PMC $p-$submanifolds (Proposition~\ref{pro:compact}), we improve their result, for $q\ge p$ and small values of $\|H\|$:  
\begin{prox}\label{pro:HII}
    For any $\varepsilon>0$, there exists $h=h(\varepsilon,p)>0$ such that
    \[\|\sff_M\|<p\cdot\min\{p-1,q\}+\varepsilon,\]
    for any complete spacelike PMC $p-$submanifold $M$ in $\hyp^{p,q}$ with mean curvature $H$ such that $\|H\|\le h$.
\end{prox}
However, for the moment we are not able to give a sharp estimates for $H\ne0$: although our techniques can be used even for non-maximal PMC $p-$submanifolds, there are terms in Bochner equation which vanish for $H=0$, and we are not able to control them yet when $H\ne0$.

\vspace{0.3cm}\noindent\textit{Intrinsic curvature.} Using similar techniques to \cite{ish88}, estimates on the scalar curvature spacelike complete PMC $p-$submanifolds in pseudo-Riemannian spaceform are given in \cite{scal}: for maximal $p-$submanifolds in the pseudo-hyperbolic space (see \cite[Theorem~1.1(2)]{scal}), they proved that
\[\Scal_M\le pq-p(p-1),\]
which is then significatively improved by Theorem~\ref{pro:scal}, when $q\ge p$. For the case $p=2$, the scalar curvature reduces to a multiple of the sectional curvature, hence we retrieve \cite{lt23}.

Using Equation~\eqref{eq:trace}, Proposition~\ref{pro:HII} can be rephrased in terms of scalar curvature for PMC $p-$submanifolds:
\begin{prox}\label{pro:Hscal}
    For any $\varepsilon>0$, there exists $h=h(\varepsilon,p)>0$
    \[\Scal_M<p\cdot\min\{0,q-p+1\}-\|H\|^2+\varepsilon,\]    
    for any complete spacelike PMC $p-$submanifold $M$ in $\hyp^{p,q}$ with mean curvature $H$ such that $\|H\|\le h$.
\end{prox}
Again, this estimate improve \cite[Theorem~1.1(2)]{scal}, for $q\ge p$ and $\|H\|$ small enough.

\subsection{Geometric relevance} To conclude the motivational context, we consider our work as another step in the direction of proving that complete maximal $p-$submanifolds are non-positively curved. 
We will rapidly recall the state of art, then show our progress in this direction, and finally motivate the importance of such result in the framework of \textit{higher-dimensional higher-rank Teichm\"uller theory}.

\vspace{0.3cm}\noindent\textit{Hyperbolicity.} The notion of \textit{Gromov-hyperbolicity} (see Definition~\ref{de:gromov-hyperbolic}) generalizes many metric features of negatively curved Riemannian manifolds for metric spaces.

We prove that complete maximal spacelike $p-$submanifolds in $\hyp^{p,q}$ which are Gromov-hyperbolic must have uniformly negative scalar curvature.


\begin{prox}\label{pro:delta}
     Let $M$ a complete maximal $p-$submanifold in $\hyp^{p,q}$. If $M$ is $\delta-$hyperbolic, then there exists a constant $c=c(\delta,p)>0$ such that 
     \[\Scal_M<-c.\]
\end{prox}

\vspace{0.3cm}\noindent\textit{State of the art.} For $p=2$, the problem is well understood by the work of \cite{bs10,lt23}: the sectional curvature of maximal surfaces is non-positive, and maximal surfaces with uniformly negative sectional curvature are fully characterized. For $q=1$, the work of \cite{andreamax} produces quantitative estimates relating sectional curvature and the width of the convex hull.

In the general case, the picture is not clear. For the moment, we have several partial results: a consequence of \cite[Theorem B]{sst23} is that maximal $p-$submanifolds in $\hyp^{p,q}$ with sufficiently regular asymptotic boundary are uniformly negatively curvated outside a compact set, hence Gromov-hyperbolic. For $q=1$, the upcoming work \cite{hshift} proves that maximal hypersurfaces in Anti-de Sitter have uniformly negative sectional curvature if contained in sufficiently small convex bodies. 

The Gromov-hyperbolicity of complete spacelike maximal $p-$submanifolds admitting a cocompact action by a discrete subgroup of $\mathrm{O}(p,q+1)$ is completely understood after the work \cite{beykas}. They give an obstruction of being Gromov-hyperbolic, which consists in containing a so-called $j-$crown in the asymptotic boundary. 

In our language, $j-$crowns are precisely boundaries of pseudo-flat submanifolds of dimension $j+1$ (see Definition \ref{de:pseudoflat}). Proposition~\ref{pro:delta} is a first generalization of their result in the non-equivariant case.

We see Theorem~\ref{thm:scal}, Theorem~\ref{thm:Ricci} and Proposition~\ref{pro:delta} as a further hint of the fairness of our claim about the curvature of maximal $p-$submanifolds, and the methods used as a stepping stone to prove it.

\vspace{0.3cm}\noindent\textit{Geometric topology.} Higher higher Teichm\"uller theory is a recent branch of geometric topology, which investigates connected components of the character variety \[\chi\left(\pi_1(M),G\right):=\mathrm{Hom}\left(\pi_1(M),G\right)/\!/\,G\] entirely consisting of discrete and faithful representations, for $M$ a closed manifold of dimension $p\ge2$ and $G$ a semi-simple Lie group of rank $r\ge1$. Such components are called \textit{higher higher Teichm\"uller spaces}. See \cite{wien} for a survey on the topic.

If $M=S_g$ is a closed surface of genus $g\ge2$ and $G=\mathbb{P}\mathrm{SL}(2,\R)=\mathrm{Isom}^+(\hyp^2)$, it is well known that \[\chi\left(\pi_1(S_g),\proj\mathrm{SL}(2,\R)\right)\] contains exactly two higher higher Teichm\"uller space, which are the two classical Teichm\"uller spaces $\mathcal{T}(S_g)$ and $\mathcal{T}(\overline{S_g})$ of $S_g$ corresponding to the different two orientations of $S_g$.

In the general setting, it is still a huge problem to understand whether $\chi\left(\pi_1(M),G\right)$ contains higher higher Teichm\"uller spaces, namely if there exist connected components entirely consisting of discrete and faithful representations.

In this framework, the role of the pseudo-hyperbolic space $\hyp^{p,q}$ arises from being a pseudo-Riemannian symmetric space for the semi-simple group $G=\mathrm{O}(p,q+1)$. As a homogeneous space,
\[\hyp^{p,q}=\mathrm{O}(p,q+1)/\mathrm{O}(p,q).\]

There is a long standing tradition, which can be traced back to the seminal work \cite{mes}, of studing a representation $\rho\colon\pi_1(M)\to\mathrm{O}(p,q+1)$ through its action on $\rho-$invariant subsets of $\hyp^{p,q}$.


A natural $\rho-$invariant subset is the so-called \textit{limit set} $\Lambda_\rho$, namely the adherence of the orbit of a (suitably chosen) point $x_0\in\hyp^{p,q}$ by the action of $\rho\left(\pi_1(M)\right)$. The limit set is well defined under further assumptions on the representation $\rho$ (see \cite{dgk} for details). Among the several $\rho-$invariant subsets which can be associated to $\Lambda_\rho$, two families stands out: convex bodies and spacelike $p-$submanifolds. Indeed, their geometric features, respectively the projective structure and the Riemannian structure, are suitable settings to studing the properties of the subgroup $\rho\left(\pi_1(M)\right)$.

The former approach, namely the study of $\rho-$equivariant convex subsets of $\hyp^{p,q}$, has been used in \cite{mes} for $\hyp^{2,1}$ and then extended to higher dimension, that is for $G=\mathrm{SO}(p,2)$, in \cite{barmer,bar15}. The general discussion has been carried out in several recent works including \cite{dgk,beykas,mazviag}.

In particular, the works \cite{dgk,beykas} show the existence of higher higher Teichm\"uller spaces in $\chi\left(\pi_1(M),\mathrm{O}(p,q+1)\right)$.

The latter one has also a rich history: in the $3-$dimensional Anti-de Sitter space $\hyp^{2,1}$, pleated convex surfaces and surfaces with constant Gaussian curvature have been studied respectively in \cite{mes} and \cite{areapres}. The study of pleated surfaces has been pursued in higher codimension, that is in $\hyp^{2,q}$, in \cite{mazviagpleated}. Finally, the importance of PMC $p-$submanifolds, and among them maximal $p-$submanifolds, is indisputable.

The first classification results for maximal surfaces in $\hyp^{2,1}$ are contained in \cite{bbz,bs10}, and they have been rapidly generalized to any dimension, that is to $\hyp^{p,1}$, for any $p\ge2$ (see \cite{abbz,bs10}) and to any value of mean curvature $H\in\R$, since in codimension $1$ having parallel mean curvature is equivalent to having \textit{constant} mean curvature (see \cite{bbz,abbz,tamb,ecrin}). In particular, the classification of CMC hypersurfaces produces an invariant foliation of the invisibility domain. The classification of maximal $p-$submanifolds has been further generalized in higher codimension, \textit{i.e.} in pseudo-hyperbolic space $\hyp^{p,q}$, for $q>1$: first for $p=2$ (see \cite{ctt,ltw22}) and finally for any value of $p$ and $q$ in \cite{sst23}.

The classification results disclose an incredible richness, opposed to the Bernstein theorems holding in pseudo-flat and pseudo-spherical Riemannian spaceforms (see \cite{cheng-yau,ish88}). Indeed, complete spacelike maximal $p-$submanifolds in $\hyp^{p,q}$ are in bijection with $1-$Lipschitz maps from $\sph^p$ to $\sph^q$ whose image contains no antipodal points.

From a higher higher Teichm\"uller perspective, the contribution of these results can be summarized by \cite[Corollary 1.2]{sst23}, which provides an obstruction to $\chi\left(\pi_1(M),\mathrm{O}(p,q)\right)$ to contain higher higher Teichm\"uller spaces, based on the topology and the geometry of $M$: if $\pi_1(M)$ admits a $P_1-$Anosov representation, then it acts properly discontinuously and cocompactly over the maximal $p-$submanifold $M_\rho$ whose asymptotic boundary is the limit set $\Lambda_\rho$, \textit{i.e.} the universal cover of $M$ is diffeomorphic to $M_\rho$. Since maximal $p-$submanifolds are proven to be diffeomorphic to $\R^p$, this gives a constraint on $M$.

Proving that the sectional curvature of maximal $p-$submanifolds is negative would constitute an even stronger constraint form $\chi\left(\pi_1(M),\mathrm{O}(p,q)\right)$ to contain higher higher Teichm\"uller spaces. Moreover, for $p=3$, it would imply that any $\hyp^{3,q}-$convex cocompact subgroup is a Kleinian group, which would promote pseudo-hyperbolic geometry to be a privileged setting to attack Cannon conjecture, which asserts that a hyperbolic group whose boundary at infinity is homoeomorphic to $\sph^2$ should be virtually (up to a finite index) a lattice in $\mathbb{P}\mathrm{SL}(2,\mathbb{C})$. 

\subsection{Main ingredients}\label{sub:ingredients} The strategy is to apply the (strong) maximum principle to the norm of the second fundamental form. Both in \cite{ish88,che94}, the authors produce Simon-type formulae to apply the Omori-Yau maximum principle. On the contrary, in \cite{lt23} the dimension $p=2$ has been exploited to use a holomorphic version of the Bochner formula.

The last approach cannot be used for $p>2$, since there is no natural complex structure over the maximal $p-$submanifold. For this reason, we first prove that PMC $p-$submanifolds have harmonic second fundamental form, in order to use a Bochner-Weitzenb\"{o}ck-type formula (Corollary~\ref{cor: sff harmonic iff PMC}). Through a careful analysis of the Weitzenb\"{o}ck curvature operator of a maximal submanifold, we modify the Bochner formula so that a maximum principle can be applied.

A compactness result (Proposition~\ref{pro:compact}) allows to conclude using the classical maximum principle, instead of the Omori-Yau maximum principle.

By the refined Bochner formula (Equation~\eqref{eq:II}), achieving the maximum implies that the submanifold has parallel second fundamental form and flat normal bundle. For this reason, we give a full classification of the $p-$submanifolds having these properties, showing that they coincide with the class of product $p-$submanifolds (Proposition~\ref{pro:parallel}). Among them, the maximum is actually achieved by the ones maximizing the number of factors, that are pseudo-flats.

\vspace{0.3cm}\noindent\textit{Bochner vs Simons.} The cornerstone of this work lies in the refined Bochner formula (Equation~\eqref{eq:II}), which allows us to compute the Laplacian of the scalar curvature, or equivalently of the norm of the second fundamental form, thanks to Equation~\eqref{eq:trace}.

Equation~\eqref{eq:II} and \cite[Equation~(3.5)]{ish88} give an exact expression of the laplacian of the norm second fundamental form of a maximal $p-$submanifold in a pseudo-Riemannian space form, using respectively Bochner-Weitzenb\"ock and Simons-Calabi equations. However, the content of Theorem~\ref{pro:II} and \cite[Theorem 1]{ish88} are quite different.  We would like to discuss the differences between the two methods.

Quite surprisingly, the refined Bochner formula we produce only depends on the intrinsic geometry of the $p-$submanifold, when the mean curvature $H$ identically vanishes: for this reason our result contains no codimensional contribution, as stated in Theorem~\ref{pro:scal} and Theorem~\ref{pro:II}. 

In other words, Bochner formula only detects pseudo-flat submanifolds. Since the rank of $\mathrm{O}(p,q+1)$ equals $\min\{p,q+1\}$, there are no pseudo-flat submanifolds for $p>q+1$ (where our bound can never be achieved). This is the case where the result of \cite{ish88} is sharp.

Conversely, the manipulation of the Simons formula produces an expression which mostly depends on the codimensional contribution. In the same flavour as in our case, yet in the opposite direction, Ishiara's result does not detect the lack of dimension: indeed, one cannot have more that $p$ factors for dimensional reasons, namely his bound cannot be achieved in higher codimension, \textit{i.e.} exactly when our bound sharpen the previous result.

It could be interesting to explore further how the two complementary approaches can be combined to extract geometric informations, and to understand why they detect different perspective while expressing the same quantities, namely the laplacian of the norm second fundamental form.

\subsection{Organization of the paper}

In Section \ref{section PseudoR spaceforms}, the necessary background on pseudo-Riemannian spaceforms is presented. In Section \ref{sec:product}, \emph{product submanifolds} are described and characterized. Section \ref{sec: harmonic forms} gives the necessary background on harmonic forms on Riemannian bundles, needed to produce the Bochner formula. Section \ref{sec:comput} is the core of this article: we manipulate the Bochner formula and prove Proposition~\ref{pro:scal} and Proposition~\ref{pro:II}. Lastly, in Section \ref{sec:applications} we apply our main results to prove Proposition~\ref{pro:HII}, Theorem~\ref{thm:Ricci}, and Proposition~\ref{pro:delta}.

\subsection*{Acknowledgements}
We are grateful to Andrea Seppi and Jérémy Toulisse for encouraging and promoting this collaboration since the earlier stage, and for several related discussions. We also want to thank Filippo Mazzoli and Gabriele Viaggi for the interest shown in our project.

The second author is particularly grateful to Jérémy Toulisse for pointing out a na\"{i}f mistake in the first draft, to Francesco Bonsante for helpful comments, and to Stefano Pigola for giving a deeper insight about Weitzemb\"ock formulas.

\section{Pseudo-Riemannian spaceforms}\label{section PseudoR spaceforms}
Let $(X,g)$ be a pseudo-Riemannian manifold, namely $X$ is a $(p+q)-$smooth manifold and $g$ is a non-degenerate symmetric $(2,0)-$tensor of signature $(p,q)$.

Tangent vectors are distinguished according to the sign of their norm: a vector $v$ in $\T X$ is said to be \textit{timelike} (resp. \textit{spacelike}, \textit{lightlike}, \textit{causal}) if $g(v,v)$ is negative (resp. positive, null or non-positive).

\begin{de}\label{de:spaceform}
    Let $p,q$ be two integers and $c$ be a real number. A \textit{pseudo-Riemannian spaceform} $\X^{p,q}(c)$ is a complete pseudo-Riemannian manifold of signature $(p,q)$ and constant sectional curvature $c$.
\end{de}

\begin{rem}
    Up to rescale the pseudo metric by a constant factor, we can assume that the curvature $c$ equals $-1$, $0$ or $1$.
\end{rem}
\begin{rem}
    We make no asumption on the topology of what we call a spaceform. Hence, there are several non isometric spaceform of a given signature and curvature, but they all share the same universal cover.
\end{rem}
Accordingly with the Riemannian case, a negatively curved pseudo-Riemannian spaceform is called \emph{pseudo-hyperbolic}, a flat pseudo-Riemannian spaceform is called \emph{pseudo-Euclidean} and a positively curved pseudo-Riemannian spaceform is called \emph{pseudo-spherical}.

\subsection{Explicit models of spaceforms}
Let $\V$ be a real vector space of dimension $p+q$, and $\Q$ a quadratic form with signature $(p,q)$. The space $\left(\V,\Q\right)$, denoted $\R^{p,q}$, is a pseudo-Euclidean spaceform of signature $(p,q)$. The quadric of timelike (respectively spacelike) unitary vectors
\begin{align*}
    \hyp^{p,q-1}&:=\left\{v\in\V\ |\ \Q(v,v)=-1\right\}\\
    \sph^{p-1,q}&:=\left\{v\in\V\ |\ \Q(v,v)=+1\right\}
\end{align*} endowed with the pseudo-Riemannian metric induced by restricting $\Q$ to its tangent space, is a pseudo-hyperbolic spaceform (respectively pseudo-spherical spaceform). Hereafter, we will consider $\V=\R^{p+q}$ and $\Q=\mathrm{I}_p\oplus(-\mathrm{I}_q)$.

The isometry groups of these spaces are the following:
\[\mathrm{Isom}(\R^{p,q})=\mathrm{O}(p,q)\rtimes\R^{p+q},\qquad \mathrm{Isom}(\hyp^{p,q-1})=\mathrm{Isom}(\sph^{p-1,q})=\mathrm{O}(p,q).\]

\begin{rem}
    A totally geodesic submanifold of signature $(n,k,d)$ in $\R^{p,q}$ is an affine subspace $W$ of dimension $n+k+d$ such that the restriction of the quadratic form $\Q$ to $W$ has signature $(n,k,d)$.
    
    Similarly, a totally geodesic submanifold of signature $(n,k,d)$ in $\hyp^{p,q-1}$ (respectively $\sph^{p-1,q}$) is a connected component of $W\cap\hyp^{p,q-1}$ (respectively $W\cap\sph^{p-1,q}$), for $W$ a linear subspace of $\R^{p,q}$ such that the restriction of the quadratic form $\Q$ to $W$ has signature $(n,k+1,d)$ (respectively $(n+1,k,d)$).
\end{rem}

\subsection{Spacelike submanifolds} A submanifold $M$ of a pseudo-Riemannian manifold $(X,g)$ is \textit{spacelike} if the induced metric is positive-definite, \textit{i.e.} $M$ carries a Riemannian structure. It follows that a pseudo-Riemannian manifold $X$ with signature $(p,q)$ contains no spacelike submanifold of dimension greater than $p$. In this work, we focus on spacelike submanifolds with maximal dimension, that is $p-$submanifolds.

\subsubsection{First and second fundamental forms}

A spacelike $p-$submanifold $M$ naturally induces an orthogonal splitting $\T X=\T M\oplus\No M$ over $M$. The pseudo-metric splits, as well, as \[g=g_{|\T M}\oplus g_{|\No M}=:g_\T\oplus g_\No.\]
Since the spacelike dimension is maximal, the quadratic form $g_\No$ is negative definite. Hence, a spacelike $p-$submanifold naturally induces a positive definite scalar product over $\T X|_M$, that is
\[\scal{\cdot}{\cdot}:=g_\T\oplus -g_\No,\ \mbox{ and }\ \|u\|^2:=\scal{u}{u}\ .\]

The \emph{second fundamental form} of $M$ is the symmetric covariant $2$-tensor on $M$ defined by
\[\sff(u,v)=\left(\nabla_uv\right)^\No\ ,\]
where $\nabla$ is the induced connection on $M$ and $()^\No$ denotes the normal component.

\subsubsection{Curvature}
The \emph{mean curvature} $H$ of $M$ is the trace, with respect to $g$, of the second fundamental form.
The convention for the \emph{curvature tensor} $R$ of $M$ is the following, for $u,v,w$ vectors fields of $M$:
\[R_{u,v}w=[\nabla_u,\nabla_v]w-\nabla_{[u,v]}w\ .\]
The \emph{sectional curvature} of a $2-plane$ in $\T M$ spanned by the orthonormal family $(u,v)$ is defined by
\[\mathrm{sec}(\mathrm{Span}(u,v))=g(R_{u,v}v,u)\ .\]

\subsubsection{Fundamental equations in a spaceform}

Submanifolds of spaceforms satisfy the following fundamental equations:
\begin{thm}[Fundamental equations of a submanifold]
    Let $M$ be a $p$-dimensional spacelike submanifold of a spaceform $\X^{p,q}(c)$ with constant sectional curvature $c$. Let $(u,v)$ be an orthonormal family in $\T_x M$. Then, the sectional curvature of $M$ in the tangent subspace spanned by $u$ and $v$ satisfies
\begin{equation}\tag{Gauss Equation}\label{Gauss equation}
    \sec\left(\mathrm{Span}(u,v)\right)=c+g_\No(\sff(u,u),\sff(v,v))-g_\No\left(\sff(u,v),\sff(u,v)\right).
\end{equation}
Let $X,Y,Z$ be vector fields on $M$, then
\begin{equation}\tag{Codazzi Equation}\label{Codazzi equation}
    (\nabla_X\sff)(Y,Z)=(\nabla_Y\sff)(X,Z).
\end{equation}
Let $X,Y$ be vector fields and $U,V$ be normal vector fields on $M$, then
\begin{equation}\tag{Ricci Equation}\label{Ricci equation}
    g_\No(R^\No_{X,Y}U , V) = g_\T(B_U(Y) , B_V(X)) - g_\T(B_U(X) , B_V(Y)).
\end{equation}
\end{thm}
This results are classic: the three equations can be found respectively as Corollary~4.6, Corollary~4.34 and Exercise~4.11 in \cite{ONEILL}.

\begin{rem}
    The second fundamental form is a symmetric $2-$form on $M$ taking values in the normal bundle, namely \[\sff\left(\cdot,\cdot\right)\in\Gamma\left(\left(\T^*M\otimes\T^*M\right)\otimes\No M\right).\]
    Equivalently, we can see it as a $1-$form taking values in the bundle \[\mathrm{Hom}\left(\T M,\No M\right)=\T^*M\otimes\No M.\] For the sake of clearness, we will denote by $\A$ the second fundamental form seen as an element of $\OTN{M}$, that is $\A(u):=\sff(u,\cdot)$.
\end{rem}
\subsubsection{Harmonicity of the second fundamental form in a spaceform}
\begin{lem}
    Let $M$ be a spacelike $p-$dimensional submanifold of a pseudo-Riemannian spaceform. The second fundamental form $\A$ of $M$, seen as an element of $\OTN{M}$, satisfies
    \[\dns \A=-\nabla^\No H,\]
    where $H$ is the mean curvature vector of $M$.
\end{lem}
\begin{proof}
    Let $\nb$ be the Levi-Civita connection of $\X^{p,q}$. Let $u$ be a vector in $\T_x M$ and $(e_k)_{1\leq k\leq p}$ a local orthonormal frame of $\T M$ around $x$, that is parallel at $x$. The adjoint of the covariant exterior derivative of $\A$ is then, at $x$
    \begin{align*}
        \dns \A(u) = & -\tr((\nb_\cdot \A)(\cdot)(u))\\
        = & -\tr((\nb_\cdot \sff)(\cdot,u))\\
        = & -\sum_{1\leq k\leq p}(\nb_{e_k}\sff)(e_k,u)\\
        = & -\sum_{1\leq k\leq p}(\nb_u\sff)(e_k,e_k) & \text{by symmetry and \ref{Codazzi equation}}\\
        = & -\sum_{1\leq k\leq p}\nabla^\No_u(\sff(e_k,e_k)) + \cancel{2\sff(\nabla_u e_k,e_k)} & \text{$\nabla_u e_k$ vanish at $x$}\\
        = & -\sum_{1\leq k\leq p}\nabla^\No_u(\sff(e_k,e_k))
    \end{align*}
\end{proof}
\begin{cor}\label{cor: sff harmonic iff PMC}
    The second fundamental form of a spacelike $p-$submanifold $M$ of $\X^{p,q}$ is a harmonic 1-form with values in $E$ if and only if $M$ has parallel mean curvature.
\end{cor}

\subsection{Two results in $\hpq$}

We explain here two results from \cite{sst23} concerning $p$-dimensional complete maximal submanifolds of $\hpq$:
\begin{itemize}
    \item the Plateau problem is uniquely solvable for a class of spheres in the boundary of $\hpq$,
    \item the isometry group of $\hpq$ acts cocompactly on the space of pointed complete maximal submanifolds of $\hpq$.
\end{itemize}

The second result is extended to the case of Parallel Mean Curvature submanifolds.

\subsubsection{Boundary of $\hpq$}

Recall that $\hpq$ is defined as being the set
\[\hpq :=\left\{v\in\V\ |\ \Q(v,v)=-1\right\}\ ,\]
where $(\V,\Q)$ is a signature $(p,q+1)$ real quadratic space.

It has then a topological boundary, called the \emph{Einstein universe}, and defined as follows:
\[\Einpq:=\{v\in\V\setminus\{0\}\ |\ \Q(v)=0\}/\R_{>0}\ .\]

The space boundary of a $p$-dimensional spacelike submanifold $M$ of $\hpq$ is then 
\[\partial M:=\overline{M}\cap\Einpq\ ,\]
where $\overline{M}$ denotes the closure of $M$ in $\hpq\cup\Einpq$.

\subsubsection{Solution to the Plateau problem}

In \cite{sst23}, Seppi, Smith and Toulisse solve the following Plateau problem. They call a topological sphere in $\Einpq$ \emph{admissible} if it bounds some complete spacelike $p$-dimensional submanifold of $\hpq$. For these admissible spheres, the Plateau problem is uniquely solvable:
\begin{thm*}[Theorem A in \cite{sst23}]
    For every admissible sphere $\Lambda$ in $\Einpq$, there is a unique complete maximal $p$-dimensional submanifold $M$ of $\hpq$ such that $\partial M$ equals $\Lambda$.
\end{thm*}

\subsubsection{Compactness of spacelike $p-$submanifolds with bounded mean curvature}
In this brief paragraph, we generalize \cite[Theorem~5.1]{sst23} to spacelike complete PMC $p-$submanifold in $\hyp^{p,q}$, in a weaker version.
\begin{pro}\label{pro:compact}
The group $\mathrm{O}(p,q+1)$ acts cocompactly on the space of pointed spacelike complete PMC $p-$submanifolds of $\hpq$ with bounded mean curvature.
\end{pro}
\begin{proof}
    The proof of \cite[Theorem~5.1]{sst23} uses the fact that maximal $p-$submanifolds are solution of an elliptic PDE and that there exists a uniform bound on the norm of the second fundamental form, only depending on the ambient manifold, namely on $p$ and $q$, which is guaranteed by \cite[Theorem 1]{ish88}.
    
    By \cite[Theorem 1]{cc93}, there exists a universal constant $C=C(p,q,h)$ bounding from above the norm of the second fundamental form of any spacelike complete PMC $p-$submanifold of $\hpq$ with mean curvature $H$ such that $\|H\|\le h$, hence the result extends.
\end{proof}

\section{Product submanifolds}\label{sec:product}
In this section we discuss a family of PMC spacelike submanifolds of $\hpq$, the \emph{product submanifolds}. They are defined as a particular embedding of product of hyperbolic spaces in $\hpq$, and we characterize them, following \cite{ish88}, as being the submanifolds with flat normal bundle and parallel second fundamental form (Section \ref{subsec parallel sff}).

\subsection{Definition}
Let $k$ be a positive integer smaller than or equal to $q+1$, and $n_1,\dots,n_k$ be positive integers such that $n_1+\dots n_k=p$. Consider the orthogonal splitting
  \[\R^{p,q+1}=\R^{n_1,1}\times\dots\R^{n_k,1}\times\R^{0,q+1-k}.\]
  Let $\hyp^{n_i}\hookrightarrow\R^{n_i,1}$ be the canonical embedding of the hyperbolic space in the Minkowski space.
\begin{de}\label{de:barbot}
  A \textit{product submanifold} in $\hyp^{p,q}$ is a $p-$dimensional submanifold defined by
  \begin{align*}
    \hyp^{n_1}\times\ldots\times\hyp^{n_k} & \to\hpq\\
    (x_1,\ldots,x_k) & \mapsto \sum_{i=1}^k \alpha_ix_i\ ,
\end{align*}
where the $\alpha_i$'s are positive coefficients satisfying
\[\sum_{i=1}^k \alpha_i^2=1\ .\]
\end{de}

\subsection{Geometry of product submanifold}

We describe in this paragraph the structure of the product submanifolds. Let $P=P(k,n_1,\ldots,n_k,\alpha_1,\ldots,\alpha_k)$ be a product submanifold.
\subsubsection{Isometry group}
The isometry group $\mathrm{O}(n_i,1)$ of each factor $\hyp^{n_i}$ of $P$ embeds diagonally in $\mathrm{O}(p,q+1)$.
The isometry group of a product submanifold hence contains the product of the isometry groups of the factors.
In particular, the product submanifolds are homogeneous.

\subsubsection{Geodesics}
Let $x=(x_1,\ldots,x_k)$ be a point in $P$.
A geodesic $\gamma$ in the factor $\hyp^{n_i}$ and starting at $x_i$ has the following parametrization
\[\gamma_i(t)=\cosh(t)x_i+\sinh(t)X_i\ ,\]
where $X_i$ is a unit norm vector in $\T_{x_i}\hyp^{n_i}$.
We obtain a unit speed geodesic $c_i$ in $P$ starting at $x$ by the following formula:
\[c_i(t):=\alpha_i\gamma_i\left(t/\alpha_i\right)+\sum_{j\neq i}\alpha_jx_j\ .\]

\subsubsection{Second fundamental form}

If $\ddot{c}_i$ denotes the second covariant derivative of $\hpq$, we obtain
\[\ddot{c}_i(t)=\frac{1}{\alpha_i}\left(\sum_{j\neq i}\alpha_j^2\right)\gamma_i\left(t/\alpha_i\right)-\sum_{j\neq i}\alpha_jx_j\ .\]
This implies that, for $X_i$ a unit vector tangent to the factor $\hyp^{n_i}$, we have
\[\sff_x(X_i,X_i)=\frac{1}{\alpha_i}\left(\sum_{j\neq i}\alpha_j^2\right)x_i-\sum_{j\neq i}\alpha_jx_j\ .\]

\begin{rem}
    The image of the second fundamental form, restricted to the tangent space to a given factor $\hyp^{n_i}$ is 1-dimensional.
    Indeed, each factor $\hyp^{n_i}$ is totally umbilical in $\hpq$ and is equidistant to a totally geodesic copy of $\hyp^{n_i}$.
\end{rem}

\subsubsection{Normal bundle}

The image of the second fundamental form, at the point $x$, is spanned by the vectors
\[N_i:=\sff_x(X_i,X_i)\ ,\]
that are parallel for the normal connection. Hence $P$ has flat normal bundle. Moreover, $P$ is included in a totally geodesic copy of $\hyp^{p,k-1}$ in $\hpq$.

We also have, for $1\leq i\neq j\leq k$,
\begin{align*}
    \scal{N_i}{N_i} & =1-\frac{1}{\alpha_i^2}\ ,\\
    \scal{N_i}{N_j} & =1\ .
\end{align*}

\subsubsection{Mean curvature vector}

The mean curvature vector of $P$ is, at the point $x=(x_1,\ldots,x_k)$
\[H_x=\sum_{i=1}^k\left(\frac{n_i}{\alpha_i}-p\alpha_i\right)x_i\ .\]
The mean curvature vector is parallel along the submanifold $P$, hence $P$ is a PMC submanifold of $\hpq$.

\subsection{Pseudo-flats}\label{section pseudo-flats}
\begin{de}\label{de:pseudoflat}
    A product submanifold that has $p$ factors of dimension 1 is called a \emph{pseudo-flat}.
\end{de}
This terminology will be justified below, as we will see that they are orbits of Cartan subgroups of $\mathrm{O}(p,q+1)$.

First remark that pseudo-flats exist in $\hpq$ if and only if $p\leq q+1$, and that they are always contained in a totally geodesic copy of $\hyp^{p,p-1}$.

\subsubsection{Barbot surfaces}

When $p=2$, pseudo-flats are the so-called \emph{Barbot surfaces} that have been studied in \cite[Section 4.7]{bar15}.

\subsubsection{Cartan subalgebra}

Let us fix a basis
\[(v_1,\ldots, v_p,w_1,\ldots, w_{q+1-p},v_{-p},\ldots,v_{-1})\]
of $\V$ in which the quadratic form $\Q$ is represented by the bloc matrix
\[\begin{pmatrix}
    0 & 0 & \frac{1}{2p}J_p \\
    0 & -\mathrm{Id}_{q+1-p} & 0\\
    \frac{1}{2p}J_p & 0 & 0
\end{pmatrix}\ ,\]
where $J_p$ is the $p\times p$ matrix defined by
\[J_p:=\begin{pmatrix}
    & & -1\\
    & \iddots & \\
    -1 & &
\end{pmatrix}\ .\]
We represent elements of $\mathrm{O}(p,q+1)$ and $\mathfrak{o}(p,q+1)$ in this basis.

The subspace $\mathfrak{a}$ of $\mathfrak{o}(p,q+1)$ defined by
\[\mathfrak{a}:=\{\mathrm{diag}(u_1,\ldots,u_p,0,\ldots,0,-u_p,\ldots,-u_1)\ |\ u_1,\ldots,u_p\in\R\}\]
is a Cartan subspace. Denoting by $\mathrm{K}$ the Killing form of $\mathfrak{o}(p,q+1)$, the map
\begin{align*}
    (\R^p,\scal{\cdot}{\cdot}) & \longrightarrow\quad(\mathfrak{a},K)\\
    u & \longmapsto \mathrm{diag}(u_1,\ldots,u_p,0,\ldots,0,-u_p,\ldots,-u_1)
\end{align*}
is a homothety, when the source is endowed with the natural scalar product.

\subsubsection{Cartan subgroups}

The connected Lie subgroup of $\mathrm{O}(p,q+1)$ with Lie algebra $\mathfrak{a}$ is
\[\AZ =\left\{a(u):=\mathrm{diag}(e^{u_1},\ldots,e^{u_p},1,\ldots,1,e^{-u_p},\ldots,e^{-u_1})\ |\ u=(u_1,\ldots,u_p)\in\R^p\right\}\ ,\]
and is a Cartan subgroup of $\mathrm{O}(p,q+1)$. One-parameter subgroups of $\AZ$ are the curves parametrized, for $u$ in $\R^p$, by
\begin{equation}\label{equation one-parameter subgroup Azero}
    \gamma_u(t)=a(tu)\ .
\end{equation}

\subsubsection{Polyhedron}

The points $v_1,\ldots,v_p,v_{-p},\ldots,v_{-1}$ are the vertices of a semi-positive polyhedron $\P$ in $\Einpq$ as we explain now.

The vectors $v_i$ represent points of $\Einpq$. We define a polyhedron $\widehat{\P}$ in $\Einpq$ as being the polyhedron with vertices $\{v_i,v_{-i} : 1\leq i\leq p\}$ and $(p-1)$-dimensional faces the convex combinations of sets of the form \[\{v_{\varepsilon(i)} : \varepsilon(i)\in\{\pm 1\},\ 1\leq i\leq p\}\ .\] More precisely,
\[\widehat{\P}:=\left\{\sum_{i=1}^p t_iv_{\varepsilon(i)}\ :\ t_i\geq 0,\ \sum_{i=1}^p t_i=1,\ \varepsilon(i)\in\{\pm 1\}\right\}/\ \R_+^*\ .\]

\begin{pro}\label{proposition polyhedron admissible sphere}
    The polyhedron $\P$ is a semi-positive sphere.
\end{pro}
Semi-positive spheres have been defined in \cite{sst23} (with the name \emph{admissible nonnegative sphere}) as being topological spheres $S$ in $\Einpq$ satisfying the following properties:
\begin{itemize}
    \item any triple of points in $S$ span a subspace that does not contain a negative definite 2-plane,
    \item if $p=2$, there is a triple of points in $S$ spanning a subspace of signature $(2,1)$,
    \item no photon is totally included in $S$.
\end{itemize}
\begin{lem}
    The polyhedron $\hat{\P}$ is a $p-1$-dimensional topological sphere, and does not contain two antipodal points.
\end{lem}
\begin{proof}
    By definition, $\P$ is homeomorphic to the polyhedron of $\R^p$ with vertices $\{\pm e_i\}$ and $p-1$-dimensional faces the simplexes spanned by set of vertices of the form $\{e_{\varepsilon(i)} : \varepsilon(i)\in\{\pm 1\},\ 1\leq i\leq p\}$. Hence it is a $p-1$-dimensional sphere. The last assertion is immediate.
\end{proof}

\begin{lem}
    Let $(p,q,r)$ be a triple of points in $\P$. Then the vector subspace of $\V$ spanned by $\{p,q,r\}$ does not contain a negative definite 2-plane.
\end{lem}
\begin{proof}
    Let $p$, $q$ and $r$ be three points in $\P$. By definition, $p$, $q$ and $r$ are semi-lines spanned by vectors $x$, $y$ and $z$, respectively, where
    \[x=\sum_{i=1}^p x_iv_{\varepsilon(i)},\quad y=\sum_{i=1}^p y_iv_{\omega(i)},\quad z=\sum_{i=1}^p z_iv_{\sigma(i)}\ ,\]
    the $x_i, y_i, z_i$ are nonnegative and sum, respectively, to 1, and $\varepsilon(i),\omega(i),\sigma(i)$ are $\pm 1$. Then
    $\scal{x}{y}$, $\scal{y}{z}$ and $\scal{x}{z}$ are nonnegative numbers. We compute the Gram determinant of the family $(x,y,z)$ to obtain
    \[\begin{vmatrix}
        0 & \scal{x}{y} & \scal{x}{z}\\
        \scal{y}{x} & 0 & \scal{y}{z}\\
        \scal{z}{x} & \scal{z}{y} & 0
    \end{vmatrix}=2\scal{x}{y}\scal{x}{z}\scal{y}{z}\leq 0\ .\]
    The subspace of $\V$ spanned by $p$, $q$ and $r$ cannot be of signature $(1,2)$ and hence does not contain a negative definite 2-plane.
\end{proof}
\begin{proof}[Proof of Proposition \ref{proposition polyhedron admissible sphere}]
    By the previous two lemmas, $\P$ is a topological sphere that does not contain photons and no negative triple. The triple $([v_1+v_2],[v_2+v_{-1}],[v_{-1}+v_{-2}])$ span a subspace of signature $(2,1)$.
\end{proof}

By \cite[Theorem A]{sst23}, there is a unique maximal surface in $\hpq$ bounded by $\P$.
We will show that this maximal surface is a pseudo-flat and is indeed an orbit of the Cartan subgroup $\AZ$.

\subsubsection{Orbits of the Cartan subgroup $\AZ$}\label{sub:cartan}

Denote by $x$ the point of $\hpq$ defined by 
\begin{equation}\label{equation definition point x}
    x:=\sum_{i=1}^p(v_i+v_{-i})\ ,
\end{equation}
and by $x_{\mu,\theta}$ the point of $\hpq$ defined, for $\mu=(\mu_1,\ldots,\mu_{q+1-p})$ a unit vector in $\R^{q+1-p}$ and $\theta$ in $[0,2\pi)$, by
\begin{equation}\label{equation definition point x mu theta}
    x_{\mu,\theta}:=\left[\cos(\theta)\sum_{i=1}^p(v_i+v_{-i})+\sin(\theta)\sum_{j=1}^{q+1-p}\mu_jw_j\right]\ .
\end{equation}

Denote by $\Sigma_{\mu,\theta}$ the orbit of the point $x_{\mu,\theta}$ by the action of $\AZ$.

\begin{pro}
    The set $\Sigma_{\mu,\theta}$ is a $p$-dimensional spacelike PMC submanifold of $\hpq$. Its space boundary consists of the polyhedron $\P$. The induced metric is flat and the mean curvature vector at $x_{\mu,\theta}$ equals
    \[H_{x_{\mu,\theta}}= \frac{\sqrt{p}\sin(\theta)^2}{\cos(\theta)} \sum_{i=1}^p(v_i+v_{-i})-p\sin(\theta)\sum_{j=1}^{q+1-p}\mu_jw_j.\]
    If $\theta$ vanishes, then $\Sigma_{\mu,0}$ is the complete maximal submanifold bounded by $\P$.
\end{pro}

\begin{rem}
The submanifold $\Sigma_{\mu,0}$ is a product submanifold with factors $\hyp^1$ hyperbolic spaces of dimension 1.
\end{rem}

\begin{rem}
    We call $\Sigma_{\mu,0}$ a \emph{pseudo-flat} by analogy with maximal flats of a Riemannian symmetric space. Remark that the submanifolds $\Sigma_{\mu,\theta}$ are not totally geodesic, but always isometric to the Euclidean space of dimension $p$.
\end{rem}

\begin{proof}
    Since $\Sigma$ is homogeneous, we only need to make computations at a single point, say at $x_{\mu,\theta}$. Given a vector $u$ in $\R^p$, we can draw the curve $c_u=\gamma_u\cdot x_{\mu,\theta}$ where $\gamma_u$ is the one-parameter subgroup of $\AZ$ defined by Equation \eqref{equation one-parameter subgroup Azero}. We denote by $(e_1,\ldots,e_p)$ the canonical basis of $\R^p$ and by $c_i$ the curve $c_{e_i}$. Given an integer $i$ between $1$ and $p$, $\dot{c}_i(0)=\cos({\theta})(v_i-v_{-i})$, hence the family
    \[\left(\frac{\sqrt{p}}{\cos(\theta)}\dot{c}_i(0)\right)_{i=1}^p\]
    is an orthonormal basis of $ \T_{x_{\mu,\theta}} \Sigma_{\mu,\theta} $ and we see that $\Sigma_{\mu,\theta}$ is spacelike.

    The submanifold $\Sigma_{\mu,\theta}$ is also a PMC submanifold because it is an orbit of a subgroup of $\mathrm{O}(p,q)$, that is the isometry group of $\hpq$.

    For a fixed $i$ between $1$ and $p$, we compute $\sff(\dot{c}_i(0),\dot{c}_i(0))$. To this aim, we compute the orthogonal projection of $\ddot{c}_i(0)$ on $\No_{x_{\mu,\theta}}\Sigma_{\mu,\theta}$. We have
    \[\ddot{c}_i(0)=\cos(\theta)(v_i+v_{-i})\ ,\]
    hence
    \[\scal{\ddot{c}_i(0)}{x_{\mu,\theta}}=-\frac{\cos(\theta)^2}{p}\quad\mbox{and}\quad\scal{\ddot{c}_i(0)}{\dot{c}_j(0)}=0\ .\]
    At the end,
    \begin{equation}\label{equation calcul II_{ii}}
    \sff(\dot{c}_i(0),\dot{c}_i(0))=\ddot{c}_i(0)^\No=\ddot{c}_i(0)-\frac{\scal{\ddot{c}_i(0)}{x_{\mu,\theta}}}{\Q(x_{\mu,\theta})} x_{\mu,\theta}= \ddot{c}_i(0)-\frac{\cos(\theta)^2}{p}x_{\mu,\theta}\ .
    \end{equation}

    Now, if $i$ and $j$ are two distinct integers between $1$ and $p$, then
    \begin{equation}\label{equation calcul II_{ij}}
    \sff(\dot{c}_i(0),\dot{c}_j(0))=0\ .
    \end{equation}
    Indeed, the vector $(v_i-v_{-i})$ is invariant under the action of the one-parameter subgroup $(\gamma_j(t))_{t\in\R}$ of $\AZ$.

    We compute the mean curvature vector by the following formula:
    \[H_{x_{\mu,\theta}} = \sum_{i=1}^p \sff \left(\frac{\sqrt{p}}{\cos(\theta)}\dot{c}_i(0),\frac{\sqrt{p}}{\cos(\theta)}\dot{c}_i(0)\right) = \frac{p}{\cos(\theta)^2} \sum_{i=1}^p \ddot{c}_i(0)^\No\ .\]

    By Equations \eqref{equation calcul II_{ii}},

    \begin{align*}
        \sum_{i=1}^p\sff(\dot{c}_i(0),\dot{c}_i(0)) & = \frac{\cos(\theta)}{\sqrt{p}}\sum_{i=1}^p(v_i+v_{-i})-\cos(\theta)^2\left[\frac{\cos(\theta)}{\sqrt{p}}\sum_{i=1}^p(v_i+v_{-i})+\sin(\theta)\sum_j \mu_jw_j\right]\\
        & = \sin(\theta)^2 \frac{\cos(\theta)}{\sqrt{p}} \sum_{i=1}^p(v_i+v_{-i})-\cos(\theta)^2\sin(\theta)\sum_{j=1}^{q+1-p}\mu_jw_j\ ,
    \end{align*}
    hence
    \[ H_x = \frac{\sqrt{p}\sin(\theta)^2}{\cos(\theta)} \sum_{i=1}^p(v_i+v_{-i})-p\sin(\theta)\sum_{j=1}^{q+1-p}\mu_jw_j\ .\]

    We see from this last formula that when $\theta$ vanishes, $\Sigma_{\mu,\theta}$ is a maximal submanifold.
    
    Now for the curvature assumption, we use \ref{Gauss equation} that says, for $i$ and $j$ two distinct integers between $1$ and $p$ and denoting $e_i=\frac{\sqrt{p}}{\cos(\theta)}\dot{c}_i(0)$,
    \[\sec(e_i,e_j)=-1+\scal{\sff(e_i,e_i)}{\sff(e_j,e_j)}-\scal{\sff(e_i,e_j)}{\sff(e_i,e_j)}\ .\]
    By Equations \eqref{equation calcul II_{ii}} and \eqref{equation calcul II_{ij}}, we obtain
    \[\sec(e_i,e_j)=-1+1-0=0\ .\qedhere\]
\end{proof}

\subsection{Parallel second fundamental form and flat normal bundle}\label{subsec parallel sff}

In this section, we prove that the class of spacelike $p-$submanifolds of $\hpq$ having parallel second fundamental form and normal flat bundle coincides with the class of product submanifolds.

\begin{lem}\label{cor:commutator}
    Let $M$ be a spacelike $p-$submanifold of a spaceform $\X^{p,q}$ of signature $(p,q)$. The following are equivalent:
    \begin{enumerate}
        \item the normal bundle is flat, that is $R^\No =0$;\label{it:normalflat}
        \item for any $U,V$ sections of $\No M$, $B_U$ and $B_V$ commute;\label{it:commute}
        \item for any $x$ in $M$, there exists an orthonormal basis $e_1,\dots,e_p$ such that the matrix of $B_U$ with respect to the basis $e_i$ is diagonal, for any section $U$ of $\No M$.\label{it:simmultdz}
    \end{enumerate}
\end{lem}
\begin{proof}
    The equivalency between \eqref{it:normalflat} and \eqref{it:commute} is due to \eqref{Ricci equation}. Indeed, since $B_U$ is a self-adjoint operator with respect to the metric $g_\T$, we get
    \begin{align*}
        g_\No\left(R^\No_{X,Y}U , V\right)&= g_\T\left(B_U(Y) , B_V(X)\right) - g_\T\left(B_U(X) , B_V(Y)\right)\\
        &= g_\T\left(B_V B_U(Y) , X\right) - g_\T\left(X , B_U B_V(Y)\right)\\
        &=g_\T\left(B_V B_U(Y)-B_U B_V(Y), X\right).
    \end{align*}
    
    The equivalency between \eqref{it:commute} and \eqref{it:simmultdz} is classical linear algebra.
\end{proof}

\begin{pro}\label{pro:parallel}
    Let $M$ be a complete spacelike $p-$submanifold in $\hpq$. If the second fundamental form $\sff$ of $M$ is parallel with respect to the normal connection and the normal bundle is flat, then $M$ is a product submanifold.
\end{pro}
\begin{rem}
    If the second fundamental form of $M$ is parallel, then clearly the mean curvature is parallel, as well.
\end{rem}

\begin{proof}
The proof is divided in four steps: first, we prove that $\nabla^\No\sff=0$ is equivalent to say that $\sff$ commutes with the parallel transport along any curve, hence we can work at a fixed point $x$ of $M$ and then extend the result along $M$. The second step consists in building parallel subbundles, which are morally the eigenspaces of the shape operator: by De Rham decomposition theorem, $M$ is product of the integral submanifolds of these subbundles. Each factor is isometric to the hyperbolic spaces of a suitable dimension, up to a constant conformal factor. Finally, the study of the image of the second fundamental form concludes the proof.
\step{$\sff$ commutes with the parallel transport}\label{step:par1}
    Let $x,y$ be points of $M$ and consider a curve $\gamma\colon[0,1]\to M$ joining them. Denote
    \[P_\gamma\colon\T_x M\to\T_y M\]
    the parallel transport along $\gamma$. We want to prove that for any pair $(v,w)$ of vectors in $\T_x M$, it holds
    \[\sff\left(P_\gamma(v),P_\gamma(w)\right)=P_\gamma\left(\sff(v,w)\right).\]

    Let $X,Y$ be the two parallel vector fields along $\gamma$ such that $X(0)=v$ and $Y(0)=w$, then $\sff(X,Y)$ is parallel along $\gamma$: indeed,
    \[\nabla^\No_{\dot{\gamma}}\left(\sff(X,Y)\right)=(\nabla^\No_{\dot{\gamma}}\sff)(X,Y)+\sff(\nabla_{\dot{\gamma}}X,Y)+\sff(X,\nabla_{\dot{\gamma}}Y),\]
    which vanishes since $\nabla^\No_{\dot{\gamma}}\sff=0$ by hypothesis, while $\nabla_{\dot{\gamma}}X=\nabla_{\dot{\gamma}}Y=0$ by construction.
    
    It follows that
    \[P_\gamma\left(\sff(v,w)\right)=\sff(X(1),Y(1))=\sff\left(P_\gamma(v),P_\gamma(w)\right),\]
    concluding the proof of Step~\ref{step:par1}.
    
    \step{Generalized eigenspaces}\label{step:par2}
    Let us fix $x$ in $M$, and consider $N_1,\dots,N_q$ an orthonormal basis of $\No_x M$. Define
    \[\sff^\alpha:=-g_\No\left(\sff(\cdot,\cdot),N_\alpha\right),\]
    namely $\sff^\alpha$ is the projection of $\sff$ along the direction $N_\alpha$. Denote by $B^\alpha:=B_{N_\alpha}$ the adjoint endomorphism of $\sff^\alpha$, namely
    \[\begin{tikzcd}[row sep=1ex]
        B^\alpha\colon\T_x M\arrow[r] & \T_x M\\
        v\arrow[r,mapsto] & \sum_{i=1}^p \sff^\alpha(v,e_i)e_i,
    \end{tikzcd}\]
    for $e_1,\dots,e_p$ an orthonormal basis of $\T_x M$.

    One can check that $B^\alpha$ is a symmetric endomorphism of $\T_x M$, hence diagonalizable. Denote $V^\alpha_1,\dots,V^\alpha_{n_\alpha}$ the eigenspaces of $B^\alpha$, and consider the set of multi-index
    \[\mathcal{I}=\{1,2,\dots,n_1\}\times\dots\times\{1,2,\dots,n_q\}.\]
    For any multi-index $I=(i_1,\dots,i_q)\in\mathcal{I}$, define $V_I:=V^1_{i_1}\cap\ldots\cap V^q_{i_q}$. 
    
    Since the normal bundle is flat, the $B^\alpha$'s are simultaneously diagonalizable (Lemma~\ref{cor:commutator}). It follows that \[\T_x M=\bigoplus_{I\in\mathcal{I}} V_I.\] One can check that such decomposition does not depend on the choice of the basis $N_\alpha$. Hence, by Step~\ref{step:par1}, one can extend each $V_I$ to a parallel subbundle of $\T M$.
    
    \step{$M$ is a product of hyperbolic spaces}
    Denote $M_I$ the integral submanifold of $V_I$ containing $x$. By De Rham decomposition theorem (see \cite[Chapter~IV, Section~6]{kobanom1}), $M_I$ is totally geodesic and $M$ is isometric to the product of the $M_I$'s. To conclude the proof, it suffices to prove that $M_I$ has negative constant sectional curvature.

    For $I\in\mathcal{I}$, denote $\lambda_I:=(\lambda^1_{i_1},\dots,\lambda^q_{i_q})$, where $\lambda^\alpha_{i_\alpha}$ is the eigenvalue of $B^\alpha$ relative to the eigenspace $V^\alpha_{i_\alpha}$. By definition, for each pair of vectors $(v,w)$ in $\T_x M_I$, one has
    \begin{equation*}
        \sff^\alpha(v,w)=g_\T\left(B^\alpha(v),w\right)=\lambda^\alpha_{i_\alpha}g_\T(v,w).
    \end{equation*}
    If $v,w$ are orthonormal, then
    \begin{equation}\label{eq:gaussparal}
    \begin{split}
    g_\No(\sff(v,w),\sff(v,w))&=-\sum_{\alpha=1}^q \sff^\alpha(v,w)\sff^\alpha(v,w)=0\\
    g_\No(\sff(v,v),\sff(w,w))&=-\sum_{\alpha=1}^q \sff^\alpha(v,v)\sff^\alpha(w,w)=-\sum_{\alpha=1}^q\lambda^\alpha_{i_\alpha}\lambda^\alpha_{i_\alpha}=-\|\lambda_I\|^2,
    \end{split}
    \end{equation}
    for $\|\cdot\|$ the standard norm of $\R^q$. Substituting in \ref{Gauss equation}, one obtains

    \[
    K(v,w)=-1+g_\No\left(\sff(v,v),\sff(w,w)\right)-g_\No\left(\sff(v,w),\sff(v,w)\right)=-1-\|\lambda_I\|^2<0,
    \]
    that is $M_I$ has constant negative sectional curvature, hence it is a copy of $\hyp^{\dim V_I}$, up to the constant conformal factor $(1+\|\lambda_I\|^2)^{-1}$ .
    \step{$M$ has the same second fundamental form as a product submanifold}
    For every unit vector $u$ in $V_I$,
    \[\sff(u,u)=\sum_\alpha(\lambda^\alpha_{i_\alpha} e^\alpha)\ ,\]
    hence the image of $\sff_{|V_I}$ is of dimension 1 and each factor $M_I$ is a totally umbilical submanifold of $\hpq$. Indeed, $M_I$ is a totally umbilic submanifold of $\R^{p,q+1}$, isometric to a hyperbolic space. Such spaces are included in copies of $\R^{p,1}$ in $\R^{p,q+1}$. In particular, $M$ is included in a copy of $\hyp^{p,k-1}$ in $\hpq$.

    If $u,v$ are unit vectors in $\T_xM_I$, $\T_xM_J$, respectively, then by \ref{Gauss equation} we have
    \[g_\No(\sff(u,u),\sff(v,v))=1\ .\]
    
    Now, denoting $N_I$ unit normal vectors in $\No_xM$, the equations
    \begin{align*}
        \scal{N_I}{N_I} & =\scal{\lambda_I}{\lambda_I}^*\\
        \scal{N_I}{N_J} & =1
    \end{align*}
    determine the vectors $N_I$ up to rotations.

    By the fundamental theorem of submanifolds in spaceforms, since the submanifold $M$ has the same metric and second fundamental form as a product submanifold, $M$ is a product submanifold.
\end{proof}


\section{Harmonic forms and Bochner formula}\label{sec: harmonic forms}
In this section, we briefly present the theory of harmonic forms with values in a Riemannian vector bundle. The goal is to deduce a Bochner formula contained in Theorem~\ref{thm:bochner}. 

Let $M$ be a Riemannian manifold, and $\left(E,\scal{\cdot}{\cdot}\right)\to M$ a Riemannian vector bundle over $M$, endowed with a metric connection $\nb$. In the next paragraphs, we define two Laplacian operators over the bundle $E$, present a Weitzenböck decomposition linking these two operators and deduce a Bochner formula. This is classical and can be found in \cite{xin96}.

\subsection{The rough Laplacian}
We have a natural Laplacian on $\Omega^1(M,E)$. Indeed, the Levi-Civita connections $\nabla$ of $M$ and the connection $\nb$ induce a connection on $\Omega^1(M,E)$, still denoted $\nb$. Taking twice the differential of $\alpha$ in $\Omega^1(M,E)$, we obtain a $2-$tensor with values in $\Omega^1(M,E)$. The \emph{rough Laplacian} of $\alpha$ is the trace of such $2-$tensor. We denote the rough Laplacian by $\tr(\nb^2)$. With this convention, the rough Laplacian naturally extends the usual Laplacian on functions, and is a semi-negative operator. If $(e_i)$ is a given local orthonormal frame for $\T M$, and $\alpha$ a 1-form on $M$ with values in $E$, then
\begin{align*}
    \tr(\nb^2(\alpha)) & =\sum_{i=1}^p \nb^2 \alpha(e_i,e_i) \\
    & = \sum_{i=1}^p(\nb_{e_i}\nb\alpha)(e_i)\\
    & = \sum_{i=1}^p\nb_{e_i}\nb_{e_i}\alpha-\nb_{\nabla_{e_i}e_i}\alpha\ .
\end{align*}

\subsection{The Hodge Laplacian}

There is another Laplacian-type operator, denoted by $\Delta$, the \textit{Hodge Laplacian} on $E$. It is defined by
\begin{align*}
    \Delta : \Omega^1(M,E)\ \to\ &  \Omega^1(M,E)\\
    \alpha \ \mapsto\ & \dns\dn\alpha+\dn\dns\alpha,
\end{align*}
where $\dn$ is the exterior covariant derivative associated with $\nb$ and $\dns$ is its formal adjoint with respect to the natural pairing on $\Omega^k(M,E)$. We recall some formulae. If $\alpha$ is a 1-form with values in $E$, we have
\begin{align*}
    \dn \alpha= & \mathcal{A}(\nb\alpha)\ ,\\
    \dns\alpha= & -\tr(\nb\alpha)\ ,\\
    \dn\dns\alpha= & -\nb\tr(\nb\alpha)\ ,
\end{align*}
where $\mathcal{A}$ denotes the antisymmetrization of a covariant $2$-tensor.

To express $\dns\dn\alpha$ we take $(e_i)_{1\leq i\leq p}$ a local orthonormal frame of $\T M$ such that each $e_i$ is parallel at the point $x$. We also take $X$ a local vector field parallel at $x$. Then, at $x$,
\begin{align*}
    \dns\dn\alpha(X) & =\sum_{i=1}^p -(\nb_{e_i}\dn\alpha) (e_i,X)\\
    & =\sum_{i=1}^p -\nb_{e_i}(\dn\alpha(e_i,X))+\dn\alpha(\nabla_{e_i}e_i,X)+\dn\alpha(e_i,\nabla_{e_i}X)\\
    & = \sum_{i=1}^p -\nb_{e_i}((\nb_{e_i}\alpha)(X)-(\nb_X\alpha)(e_i)-\alpha([e_i,X]))\\
    & = \sum_{i=1}^p -(\nb_{e_i}\nb_{e_i}\alpha)(X)+(\nb_{e_i}\nb_X\alpha)(e_i).
\end{align*}

\subsection{The Weitzenböck decomposition of $\Delta$}

We denote by $\bar{R}$ the curvature tensor of the bundle $\mathcal{A}^1(M,E)$ endowed with the tensor product connection $\nb$. We denote by $S$ the following operator on 1-forms with values in $E$:
\[S(\alpha)(X):=\sum_{i=1}^p(\bar{R}_{e_i,X}\alpha)(e_i)\ ,\]
for every vector field $X$ on $M$ and for $(e_i)_{1\leq i\leq p}$ an orthonormal frame of $\T M$.
\begin{pro}[Weitzenböck decomposition]
    Let $\alpha$ be an element of $\Omega^1(M,E)$. Then
    \[\Delta\alpha=-\tr(\nb^2\alpha)+S(\alpha)\ .\]
\end{pro}
\begin{proof}
    Let $\alpha$ be a section of $\mathcal{A}^1(M,E)$ and $X$ be a vector field on $M$, parallel at $x$. Again, $(e_i)_{1\leq i\leq p} $ is a local orthonormal frame, parallel at $x$.
    
    By the above formulae, we have, at $x$
    \begin{align*}
        \Delta\alpha (X) & = \sum_{i=1}^p-(\nb_{e_i}\nb_{e_i}\alpha)(X)+(\nb_{e_i}\nb_X\alpha)(e_i)-(\nb_X\nb_{e_i}\alpha)(e_i)\\
        & = -\tr(\nb^2\alpha)(X)+\sum_{i=1}^p +(\nb_{e_i}\nb_X\alpha)(e_i)-(\nb_X\nb_{e_i}\alpha)(e_i)\\
        & = -\tr(\nb^2\alpha)(X) + \sum_{i=1}^p (\bar{R}_{e_i,X}\alpha)(e_i)
    \end{align*}
\end{proof}

\subsection{The Bochner formula}

Take $\alpha$ a $1$-form on $M$ with values in $E$. We can compute
\[\Delta\|\alpha\|^2,\]
where $\Delta$ is the usual Laplacian on functions, that coincides with the rough Laplacian on sections of the trivial line bundle $M\times\R$. Using the Weitzenböck decomposition, we obtain
\begin{align*}
    \frac{1}{2} \Delta \|\alpha\|^2 & = \frac{1}{2}\tr (D^2\scal{\alpha}{\alpha})\\
    & = \tr (D \scal{\nb\alpha}{\alpha})\\
    & = \tr \scal{\nb^2\alpha}{\alpha}+\tr \scal{\nb\alpha}{\nb\alpha}\\
    & = \scal{S(\alpha)-\Delta\alpha}{\alpha}+\scal{\nb\alpha}{\nb\alpha}.
\end{align*}
We deduce from this computation the following:

\begin{thm}[Bochner formula]\label{thm:bochner}
    Let $\alpha$ be a harmonic 1-form on $M$ with values in $E$. Then 
    \[\frac{1}{2}\Delta\|\alpha\|^2=\|\nb\alpha\|^2+\langle S(\alpha),\alpha\rangle\ ,\]
    where $\Delta$ is the usual Laplacian on functions.
\end{thm}

\section{Bochner formula applied to maximal submanifolds}\label{sec:comput}
Let $M$ be a spacelike $p-$submanifold in $\X^{p,q}$, and let $(e_i)_{1\leq i\leq p+q}$ be an orthonormal frame of $\T\X^{p,q}$ along $M$, adapted to $M$ in the following sense:
\begin{itemize}
    \item the vector fields $(e_i)_{1\leq i\leq p}$ are tangent to $M$,
    \item the vector fields $(e_\alpha)_{p+1\leq \alpha\leq p+q}$ are normal to $M$.
\end{itemize}
We use the convention that latin indices $i,j,k$ range from $1$ to $p$ and greek indices $\alpha,\beta,\gamma$ range from $p+1$ to $p+q$.

\begin{de}
    Let $(e_i)_{1\leq i\leq p+q}$ be an orthonormal frame of $\X^{p,q}$ adapted to $M$. The second fundamental form writes as
    \[\sff=\sum h_{ij}^\alpha dx^i\otimes dx^j\otimes e_\alpha.\] We denote $H^\alpha$ the matrix $(h_{ij}^\alpha)_{ij}$ in $\mathcal{M}(p,\R)$, for $\alpha=p+1,\dots, p+q$, namely $H^\alpha=-\scal{\sff}{e_\alpha}$ written in coordinates $e_1,\dots,e_p$.
\end{de}

\subsection{The Laplacian of the scalar curvature}
The main goal of this section is to prove the following formula, from which we deduce the results of this article.
\begin{pro}\label{proposition formula}
    Let $M$ be a complete maximal spacelike $p-$submanifold in a pseudo-Riemannian spaceform $\X^{p,q}(c)$, then
    \begin{equation}\label{eq:II}
    \begin{split}
        \frac{1}{2}\Delta\|\sff\|^2=&\|\nabla\sff\|^2+\sum_{\alpha,\beta=1}^q\|[H^\alpha,H^\beta]\|^2+\sum_{i\ne j=1}^p\sec(e_i,e_j)^2-c\,p\,\Scal+\\
        &+\sum_{i=1}^p\ric(e_i,e_i)^2+\sum_{(i,j)\ne(l,k)}(R_{ijk}^l)^2.
    \end{split}
    \end{equation}
\end{pro}

We proved in Corollary~\ref{cor: sff harmonic iff PMC} that PMC spacelike $p-$submanifolds have harmonic second fundamental form, hence we can apply the Bochner formula described in Theorem~\ref{thm:bochner}. The next paragraphs are devoted to manipulate the formula to obtain geometric information from it.

\subsection{Notations}
Let $(\X^{p,q},g)$ be a pseudo-Riemannian space form, and let $M$ be a spacelike $p-$submanifold of $\X^{p,q}$. We introduce the following notations:
\begin{itemize}
    \item $E$ the vector bundle over $M$ of homomorphisms from the tangent space $\T M$ to the normal space $\No M$,
    \item $\A$ the second fundamental form of $M$ seen as a 1-form with values in $E$,
    \item $\sff$ the second fundamental form of $M$ seen as a symmetric 2-tensor with values in $\No M$,
    \item $H$ the mean curvature vector field, defined as being the trace of $\sff$ with respect to $g_\T$. 
    \item $(e_i)_{i=1}^{p+q}$ is an orthonormal frame of $\T\X^{p,q}_{|M}$ adapted to $M$.
\end{itemize}
The vector bundle $E$ can be rewritten $\T^*M\otimes\No M$. We also have the identification \[\ATN{M}=\T^*M\otimes\T^*M\otimes \No M.\] We introduce the following connections on various bundles over $M$:
\begin{itemize}
    \item $\nabla$ is the Levi-Civita connection of $M$, hence a connection on $\T M$,
    \item $\nabla$ also denotes the dual connection of $\nabla$, hence a connection on $\T^*M$,
    \item $\nabla^\No$ is the normal connection of $M$, hence a connection on $\No M$,
    \item $\nb=\nabla^{\otimes k}\otimes\nabla^\No$ denotes a connection on $\T^*M^{\otimes k}\otimes\No M$.
\end{itemize}

The context will always impose what connection we are considering.

\subsection{Bochner formula}
In Theorem \ref{thm:bochner}, we proved a Bochner formula that can be applied to the second fundamental form $A$, being a harmonic 1-form with values in the unitary bundle $(\mathrm{Hom}(\T M,\No M),\scal{\cdot}{\cdot},\nabla)$. We obtain the following
\begin{equation}
\frac{1}{2}\Delta \|A\|^2=\|\nabla A\|^2+\scal{S(A)}{A}\ ,
\end{equation}
where $S$ is the curvature operator defined by the following formula, for $X$ a vector field on $M$:
\[S(A)(X):=\sum_{i=1}^p\left(\bar{R}_{e_i,X}A\right)(e_i)\ .\]
\subsection{Explicit computations for $\scal{S(\A)}{\A}$}
We now want to compute 
\begin{equation}\label{eq:SA1}
    \scal{S(\A)}{\A}=\sum_{i,j,k}-g_\No\left((\bar{R}_{e_k,e_i}A)(e_k)(e_j)\ ,\ A(e_i)(e_j)\right)\ .
\end{equation}

For this, we need the following general fact. Let $F$ be a vector bundle over $M$ and $\nabla^F$ be a connection on $F$. Denote by $\nabla^{*F}$ the connection induced on $\T^*M\otimes F$ by the Levi-Civita connection $\nabla$ on $\T^*M$ and $\nabla^F$ on $F$. Denote by $R, R^F, R^{*F}$ the curvature tensors of $\nabla$ on $\T M$, $\nabla^F$ on $F$ and $\nabla^{*F}$ on $\T^*M\otimes F$, respectively. Then we have the following formula, for a section $T$ of $\T^*M\otimes F$ and vector fields $X,Y,Z$ on $M$:
    \[(R^{*F}_{X,Y}T)(Z)=R^F_{X,Y}(T(Z))-T(R_{X,Y}Z).\]
Applying twice this formula yields the following.
    \[(\bar{R}_{e_k,e_i}A)(e_k)(e_j)=R^\No _{e_k,e_i}(\sff(e_k,e_j))-\sff(R_{e_k,e_i}e_k,e_j)-\sff(e_k,R_{e_k,e_i}e_j).\]

Combining it with Equation~\eqref{eq:SA1}, we obtain 
\begin{equation}\label{eq:SA2}
    \scal{S(\A)}{\A}=\sum_{i,j,k}\scal{R^\No _{e_k,e_i}(\sff(e_k,e_j))-\sff(R_{e_k,e_i}e_k,e_j)-\sff(e_k,R_{e_k,e_i}e_j)}{\sff(e_i,e_j)}.
\end{equation}

We divide this in three parts to be computed separately. Namely,
\begin{align}\tag{Part 1}\label{equation part 1}
\sum_{i,j,k}\scal{R^\No _{e_k,e_i}(\sff(e_k,e_j))}{\sff(e_i,e_j)}\ ,\\
\tag{Part 2}\label{equation part 2}
\sum_{i,j,k}\scal{\sff(R_{e_k,e_i}e_k,e_j)}{\sff(e_i,e_j)}\ ,\\
\tag{Part 3}\label{equation part 3}
\sum_{i,j,k}\scal{\sff(e_k,R_{e_k,e_i}e_j)}{\sff(e_i,e_j)}\ .
\end{align}

\subsubsection{Shape operator and second fundamental form}

\begin{lem}\label{lemma link g(B,B) with g(II,II)}
Let $U,V$ be normal vector fields and $X,Y$ be tangent vector fields on $M$. Then
\[g_\T(B_U(X),B_V(Y))=\sum_{l=1}^pg_\No(U,\sff(X,e_l))g_\No(V,\sff(Y,e_l))\ .\]
\end{lem}
\begin{proof}
It is a straightforward computation using
\[B_U(X)=\sum_{k=1}^pg_\T(B_U(X),e_i)e_i\ ,\]
and
\[g_T(B_U(X),Z)=g_\No(U,\sff(X,Z))\ .\qedhere\]
\end{proof}

\subsubsection{Computation of \eqref{equation part 1}}

\begin{lem}
\[\sum_{i,j,k}\scal{R^\No _{e_k,e_i}(\sff(e_k,e_j))}{\sff(e_i,e_j)}=\frac{1}{2}\sum_{\alpha,\beta}\|[H^\alpha,H^\beta]\|_2^2 .\]
\end{lem}
\begin{proof}
By \ref{Ricci equation}, given $i,j,k$ fixed indices,
\[-g_\No\left(R^\No_{e_k,e_i}(\sff(e_k,e_j)),\sff(e_i,e_j)\right)=-g_\T\left(B_{\sff(e_k,e_j)}(e_i),B_{\sff(e_i,e_j)}(e_k)\right)+g_\T\left(B_{\sff(e_k,e_j)}(e_k),B_{\sff(e_i,e_j)}(e_i)\right)\ .\]
Applying Lemma \ref{lemma link g(B,B) with g(II,II)} gives
	\begin{align*}
        g_\T\left( B_{\sff(e_k,e_j)}(e_i) , B_{\sff(e_i,e_j)}(e_k)\right)&=\sum_{l=1}^p g_\No \left( \sff(e_k,e_j),\sff(e_i,e_l)\right) g_\No\left(\sff(e_i,e_j),\sff(e_k,e_l)\right)\\
        &=\sum_{l=1}^p\sum_{\alpha,\beta=1}^q h_{kj}^\alpha h_{il}^\alpha h_{ij}^\beta h_{kl}^\beta\\
        &=\sum_{\alpha,\beta}(H^\alpha H^\beta)_{ik}h_{kj}^\alpha h_{ij}^\beta\ .\qquad \text{($H^\beta$ is symmetric)}
    \end{align*}
Now, taking the sum, we obtain
\begin{align*}
\sum_{i,j,k}g_\T\left( B_{\sff(e_k,e_j)}(e_i) , B_{\sff(e_i,e_j)}(e_k)\right) & = \sum_{i,k}\sum_j g_\T\left( B_{\sff(e_k,e_j)}(e_i) , B_{\sff(e_i,e_j)}(e_k)\right) \\
& = \sum_{i,k} \sum_{\alpha,\beta}(H^\alpha H^\beta)_{ik}(H^\beta H^\alpha)_{ik}\quad\text{($H^\alpha$ is symmetric)}\\
& = \sum_{\alpha,\beta}\scal{H^\alpha H^\beta}{H^\beta H^\alpha}_2\ .
\end{align*}

Similarly,
\[\sum_{i,j,k}g_\T\left(B_{\sff(e_k,e_j)}(e_k),B_{\sff(e_i,e_j)}(e_i)\right)=\sum_{\alpha,\beta}\| H^\alpha H^\beta\|^2_2\ .\]

This finishes the proof, since
\[\|[H^\alpha,H^\beta]\|^2_2=2\left(\|H^\alpha H^\beta\|_2^2-\scal{H^\alpha H^\beta}{H^\beta H^\alpha}_2\right)\ .\qedhere\]
\end{proof}

\subsubsection{Computation of \eqref{equation part 2}}

\begin{lem}
\[\sum_{i,j,k}\scal{\sff(R_{e_k,e_i}e_k,e_j)}{\sff(e_i,e_j)}=c\,(p-1)\,\Scal-\sum_i \ric(e_i,e_i)^2\ .\]
\end{lem}

\begin{proof}
By the same arguments as in Lemma \ref{lemma link g(B,B) with g(II,II)}, for every $i,j,k$,
\[g_\No(\sff(R_{e_k,e_i}e_k,e_j),\sff(e_i,e_j))=\sum_l g_\T(R_{e_k,e_i}e_k,e_l)g_\No(\sff(e_i,e_j),\sff(e_j,e_l))\ .\]
Summing over $k$, we obtain
\[\sum_k g_\No(\sff(R_{e_k,e_i}e_k,e_j),\sff(e_i,e_j))=-\sum_l \ric(e_i,e_l)g_\No(\sff(e_i,e_j),\sff(e_j,e_l))\ .\]
The formula we want to prove is tensorial, hence it suffices to prove it at a point $x$ of $M$. Since $\ric_x$ is a $g_\T$-symmetric covariant 2-tensor, we can chose a basis $(e_i)_{i=1}^p$ that diagonalizes it, namely $\ric(e_i,e_j)_x=0$ for every distinct indices $i$ and $j$. We then have
\[\sum_k g_\No(\sff(R_{e_k,e_i}e_k,e_j),\sff(e_i,e_j))=-\ric(e_i,e_i)g_\No(\sff(e_i,e_j),\sff(e_j,e_i))\ .\]
By \ref{Gauss equation},
\[\sum_jg_\No(\sff(e_i,e_j),\sff(e_j,e_i))=-\ric(e_i,e_i)+c\,(p-1)\ ,\]
hence
\begin{align*}
\sum_{i,j,k}g_\No(\sff(R_{e_k,e_i}e_k,e_j),\sff(e_i,e_j)) & =\sum_i (\ric(e_i,e_i)^2-c\,(p-1)\ric(e_i,e_i))\\
& = \sum_i\ric(e_i,e_i)^2-c\,(p-1)\,\Scal\ .\qedhere
\end{align*}
\end{proof}

\subsubsection{Computation of \eqref{equation part 3}}

\begin{lem}
\[\sum_{i,j,k}\scal{\sff(R_{e_k,e_i}e_j,e_k)}{\sff(e_i,e_j)}=-\sum_{i\neq k}sec(e_i,e_k)^2-\sum_{(i,j)\neq(l,k)}\scal{R_{e_k,e_i}e_j}{e_l}^2+c\,\Scal\ .\]
\end{lem}
\begin{proof}
By the same arguments as in Lemma \ref{lemma link g(B,B) with g(II,II)}, for every $i,j,k$,
\[g_\No(\sff(R_{e_k,e_i}e_j,e_k),\sff(e_i,e_j))=\sum_l g_\T(R_{e_k,e_i}e_j,e_l)g_\No(\sff(e_i,e_j),\sff(e_k,e_l))\ .\]
Since $g_\T\left(R_{X,Y}\cdot,\cdot\right)$ is skew-symmetric 
\begin{align*}
        &\sum_{j,l=1}^p g_\T\left(R_{e_k,e_i}e_j,e_l\right) g_\No\left(\sff(e_i,e_j),\sff(e_k,e_l)\right)\\
        &=\sum_{j<l} g_\T\left(R_{e_k,e_i}e_j,e_l\right) g_\No\left(\sff(e_i,e_j),\sff(e_k,e_l)\right) + \sum_{j>l} g_\T\left(R_{e_k,e_i}e_j,e_l\right) g_\No\left(\sff(e_i,e_j),\sff(e_k,e_l)\right)\\
        &=\sum_{j<l} g_\T\left(R_{e_k,e_i}e_j,e_l\right) g_\No\left(\sff(e_i,e_j),\sff(e_k,e_l)\right) + \sum_{j<l} g_\T\left(R_{e_k,e_i}e_l,e_j\right) g_\No\left(\sff(e_i,e_l),\sff(e_k,e_j)\right)\\
        &=\sum_{j<l} g_\T\left(R_{e_k,e_i}e_j,e_l\right)\big(g_\No\left(\sff(e_i,e_j),\sff(e_k,e_l)\right)-g_\No\left(\sff(e_i,e_l),\sff(e_k,e_j)\right)\big).
\end{align*}
Both factors in the last line are skew-symmetric with respect to the variables $(i,k)$. This implies that the formula is invariant by permutation of $(i,k)$, and vanishes for $k=i$. Then,
    \begin{align*}
    &\sum_{i,k=1}^p\sum_{j<l} g_\T\left(R_{e_k,e_i}e_j,e_l\right)\big(g_\No\left(\sff(e_i,e_j),\sff(e_k,e_l)\right)-g_\No\left(\sff(e_i,e_l),\sff(e_k,e_j)\right)\big)\\
    &=2\sum_{i<k}\sum_{j<l} g_\T\left(R_{e_k,e_i}e_j,e_l\right)\big(g_\No\left(\sff(e_i,e_j),\sff(e_k,e_l)\right)-g_\No\left(\sff(e_i,e_l),\sff(e_k,e_j)\right)\big)
    \end{align*}
    
We distinguish two cases: first, consider $i=j<k=l$. In this case,
\begin{align*}
        &2\sum_{i=j<k=l}g_\T\left(R_{e_k,e_i}e_j,e_l\right)\big(g_\No\left(\sff(e_i,e_j),\sff(e_k,e_l)\right)-g_\No\left(\sff(e_i,e_l),\sff(e_k,e_j)\right)\big)\\
        &=2\sum_{i<k}\left( sec(e_i,e_k)^2-c\,sec(e_i,e_k)\right) = \sum_{i\ne k}sec(e_i,e_k)^2-c\,\Scal\ .
\end{align*}
    
Secondly, consider $i\neq j$ or $k\neq l$. By \ref{Gauss equation},
\[g_\No\left(\sff(e_i,e_j),\sff(e_k,e_l)\right)-g_\No\left(\sff(e_i,e_l),\sff(e_k,e_j)\right) =g_\T\left(R_{e_k,e_i}e_j,e_l\right)-g_\T\left(\bar{R}_{e_k,e_i}e_j,e_l\right)\ .\]
In particular, since $\X^{p,q}$ has constant sectional curvature c, for $i\neq k$ and $j\neq l$, we have
    \[\bar{R}_{e_k,e_i}e_j,e_l=c\,\delta_{ij}\delta_{lk}\ .\]
Hence,
    \begin{align*}
        g_\T\left(R_{e_k,e_i}e_j,e_l\right)\big(g_\T\left(R_{e_k,e_i}e_j,e_l\right)-g_\T\left(\bar{R}_{e_k,e_i}e_j,e_l\right)\big)
=g_\T\left(R_{e_k,e_i}e_j,e_l\right)^2\ ,
    \end{align*}
concluding the proof.
\end{proof}

We recall that the restriction of $\scal{\cdot}{\cdot}$ to $\No M$ equals to $-g_\No$, hence combining the computations of \eqref{equation part 1}, \eqref{equation part 2} and \eqref{equation part 3}, we have proven Proposition~\ref{proposition formula}.

\subsection{Maximum principle} We focus on the pseudo-hyperbolic space $\hyp^{p,q}$, that is hereafter $c=-1$. The proofs of Theorem~\ref{pro:scal} and Theorem~\ref{pro:II} are equivalent. Indeed, by tracing twice \ref{Gauss equation}, we have that the scalar curvature and the second fundamental form of spacelike $p-$submanifolds in $\hyp^{p,q}$ are linked by the following relation
\[\Scal=-p(p-1)+\|\sff\|-\|H\|^2.\]
For this reason, we only prove Theorem~\ref{pro:scal}.

We prove the first part of Theorem~\ref{pro:scal}, that we recall now:
\begin{repthmx}{pro:scal}
    Let $M$ be a complete maximal $p-$submanifold in $\hpq$. Then its scalar curvature $\Scal_M$ is non-positive. If $\Scal_M(x)=0$ at a point $x$, then $\Scal_M\equiv0$ and $M$ is a pseudo-flat submanifold.
\end{repthmx}

The proof goes in two steps:\\
First, we use a strong maximum principle, allowed by the compactness of pointed maximal $p$-submanifolds \cite[Proposition 5.1]{sst23}. This shows that the scalar curvature of $M$ is either negative or constantly zero. In the last case, this implies that the normal bundle is flat and that the second fundamental form is parallel.\\
Second, we use our study of product submanifolds to show that if $M$ is scalar-flat, $M$ must be a pseudo-flat.
\begin{proof}
\textbf{Step 1.} By \cite[Proposition 5.1]{sst23}, the group $\mathrm{O}(p,q+1)$ acts cocompactly on the space of pointed $p$-dimensional maximal submanifolds. Hence the continuous function $(M,x)\mapsto\Scal_M(x)$ achieves its maximum at a pointed maximal $p$-submanifold $(M_0,x_0)$.

The second fundamental form of a maximal submanifold is harmonic by Proposition \ref{cor: sff harmonic iff PMC}. We can hence apply the Bochner formula to the norm of the second fundamental form. By \ref{Gauss equation} and Proposition \ref{proposition formula}, we have
\[\Delta\Scal_{M_0}\geq 2p\Scal_{M_0}\ .\]

By the strong maximum principle \cite[Theorem 3.5]{gt}, either $\Scal_{M_0}(x_0)$ is negative or $\Scal_{M_0}\equiv 0$. In particular, the scalar curvature of every maximal $p$-submanifold of $\hpq$ is nonpositive.

\textbf{Step 2.} Suppose that $M$ is a pointed maximal $p$-submanifold such that its scalar curvature is identically $0$. By Proposition \ref{proposition formula}, for every $i\neq j$ and $\alpha\neq \beta$,
\[\|\nabla\sff\|=0,\quad \|[H^\alpha,H^\beta]\|,\quad sec(e_i,e_j)=0\ ,\]
so $M$ is flat and has parallel second fundamental form. Moreover, $M$ has flat normal bundle. Indeed, the terms $[H^\alpha,H^\beta]$ vanishing implies that the shape operators commute, and we conclude with Lemma \ref{cor:commutator}.

Finally, by Proposition \ref{pro:parallel}, $M$ is a product submanifold. The only flat product submanifolds being the pseudo-flats, we conclude that $M$ is a pseudo-flat.
\end{proof}

\section{Applications}\label{sec:applications}
PMC $p-$submanifolds are more difficult to deal with than maximal ones. For that reason, we are not able to give a sharp bound on the second fundamental form of spacelike complete PMC $p-$submanifolds. However, by compactness, we manage to improve the bound given by \cite{cc93} (Proposition \ref{pro:HII}). 

The scalar curvature is a weaker invariant than the Ricci curvature or the sectional curvature. However, the rigidity part of Theorem~\ref{thm:scal} allows to extract informations about those stronger invariants from the scalar curvature of a maximal $p-$submanifold in the pseudo hyperbolic space $\hyp^{p,q}$. The second application shows that in codimension $1$, the sign of scalar curvature determines the sign of the Ricci tensor (\ref{thm:Ricci}).

The third application consists in studying the $\delta-$hyperbolicity and the sectional curvature of complete spacelike maximal $p-$submanifolds in $\hyp^{p,q}$: we show that spacelike complete PMC $p-$submanifolds with small scalar curvature are $\delta-$hyperbolic (Proposition~\ref{lem:delta}), while maximal $p-$submanifolds with big scalar curvature are not $\delta-$hyperbolic (Proposition~\ref{pro:delta}). Unfortunately, the scalar curvature is an invariant too weak to completely characterize the hyperbolicity of maximal $p-$submanifolds (Remark~\ref{rem:delta})

\subsection{PMC submanifolds} We prove Proposition~\ref{pro:HII}, which is a corollary of Theorem~\ref{thm:II} and Proposition~\ref{pro:compact}.

\begin{repprox}{pro:HII}
For every positive number $\varepsilon$, there exists a positive constant $h$ depending on $\varepsilon$ and $p$ with the following properties. For any complete spacelike PMC $p-$submanifold in $\hyp^{p,q}$ with mean curvature $H$ such that $\|H\|\le h$, then
    \[\|\sff\|^2_{M_n}<p(p-1)+\varepsilon.\]
\end{repprox}
\begin{proof}
    Let us fix $\varepsilon>0$. By contradiction, assume that for any $n\in\mathbb{N}$, there exists a spacelike complete PMC $p-$submanifold $M_n$ with mean curvature $H_n$ such that
    \[\|H_n\|<\frac{1}{n},\qquad \|\sff\|^2\ge p(p-1)+\varepsilon. \]
    In particular, for any $n$, we can pick $x_n\in M_n$ such that 
    \begin{equation}\label{eq:HII}
    \|\sff_{M_n}(x_n)\|^2> p(p-1)+\frac{\varepsilon}{2}.
    \end{equation}
    Finally, let us fix a pointed totally geodesic spacelike $p-$submanifold $(x_0,H)$. For any $n\in\N$, let $\phi_n\in\mathrm{Isom}(\hyp^{p,q})$ an isometry such that 
\[\phi_n(x_n)=x_0,\qquad d_{x_n}\phi_n(T_{x_n}M_n)=H.\] 
By Proposition~\ref{pro:compact}, up to extract a subsequence, we can assume \[M_\infty:=\lim_{n\to+\infty}\phi_n(M_n)\] is a complete spacelike PMC $p-$submanifold with mean curvature $H_\infty=\lim_{n\to+\infty} H_n$. 

By continuity, Equation~\eqref{eq:HII} tells us that the norm of the second fundamental form of $M_\infty$ at $x_\infty$ is strictly greater than $p(p-1)$. However, we chose $\|H_n\|<1/n$, hence $M_\infty$ is a maximal $p-$submanifold: this contradicts Theorem~\ref{thm:II}, concluding the proof.
\end{proof}

\subsection{Ricci curvature}
The Lorentzian pseudo-hyperbolic space $\hyp^{p,1}$ is also known as \textit{Anti-de Sitter space}. In Anti-de Sitter spaces, a direct application of Theorem~\ref{thm:scal} is that the Ricci tensor of a maximal hypersurface is negative semi-definite.
\begin{repthmx}{thm:Ricci}
Let $M$ be a properly embedded spacelike maximal hypersurface in the Anti-de Sitter space $\hyp^{p,1}$, then $M$ is Ricci non-positive.

Moreover, if $\ric$ vanishes at a point $v\in\T M$, then $M$ is a product hypersurface $\hyp^{p-1}\times\hyp^{1}$.
\end{repthmx}
\begin{proof}
    Let us fix $x$ in $M$ and consider a unitary vector $v$ in $\T_x M$. By tracing \ref{Gauss equation}, we obtain for $(v_i)_{i=2}^p$ an orthonormal frame of $v^\perp$ in $\T_xM$
    \begin{equation}\label{eq:ricv}
    \ric(v,v)=-(p-1)+\|\sff(v,v)\|^2+\sum_{i=2}^p\|\sff(v,v_i)\|^2\le-(p-1)+\|\sff\|^2\le0,
    \end{equation}
    where the last inequality follows from Theorem~\ref{thm:II}.

    Moreover, if the bound is achieved at $x$, that is $\|\sff_x\|^2=p-1$, the rigidity part of Theorem~\ref{thm:II} implies that $M$ is a product hypersurface.

    One can explicitly compute the shape operator of a maximal product hypersurface $\hyp^{k}\times\hyp^{n-k}$ (see for example \cite[Equation~(18)]{ecrin}), which turns out to be
    \[B=\sqrt{\frac{p-k}{k}}\mathrm{Id}_{k}-\sqrt{\frac{k}{p-k}}\mathrm{Id}_{p-k}.\]

    To conclude, remark that, for any orthonormal basis $v_i$ of $T_x M$, we have
    \[\sum_{i=1}^p\|\sff(v,v_i)\|^2=\|B(v)\|^2\le\|B\|^2_2\|v\|^2,\]
    for $\|\cdot\|_2$ the operator norm.
    
    Comparing with Equation~\eqref{eq:ricv}, the maximum of the Ricci tensor is achieved at the eigenvector relative to the highest eigenvalue of $B$. Without loss of generality, we assume $k\le p/2$, so that the maximum eigenvalue is $(p-k)/k$. Substituting in Equation~\eqref{eq:ricv}, we conclude that
    \[\sup_{T\left(\hyp^k\times\hyp^{p-k}\right)}\ric(v,v)=-(p-1)+\frac{p-k}{k},\]
    which vanishes exactly when $k=1$.
\end{proof}

\subsection{Hyperbolicity} We first recall the basic definitions needed in order to understand the result. A length metric space $(X,d)$ is \textit{geodesic} if for each pair $(x,y)$ of points in $X$, the distance between $x$ and $y$ can be realized by a path, that is there exists an isometry $c\colon\left[0,d(x,y)\right]\to X$ such that $c(0)=x$ and $c\left(d(x,y)\right)=y$. Such path is called a \textit{minimizing geodesic}. Clearly, any connected complete Riemannian manifold is a geodesic metric space for the path distance.

Let $x_1,x_2,x_3$ be three points in a geodesic metric space $(X,d)$. A \textit{geodesic triangle} $T$ with vertex $x_1,x_2,x_3$ is the union of three minimizing geodesics $c_{ij}$ connecting $x_i$ to $x_j$, for $i\ne j$. Such geodesics are called \textit{edges} of the $T$. 

A geodesic triangle $T=c_{12}\cup c_{23}\cup c_{31}$ is $\delta-$thin, for $\delta>0$, if each edge is contained in the $\delta-$neighborhood of the union of the other two. In other words, for each $x\in c_{12}$, there exists $y\in c_{23}\cup c_{31}$ such that $d(x,y)\le\delta$, and the same for the other two edges.
\begin{de}\label{de:gromov-hyperbolic}
    Let $\delta>0$. A geodesic metric space $(X,d)$ is $\delta-$hyperbolic if any geodesic triangle is $\delta-$thin.
\end{de}

\begin{lem}\label{lem:deltaconv}
Let $M_n$ be a sequence of $\delta-$hyperbolic metric space. If $M_\infty$ is the limit of $M_n$ in the Gromov-Hausdorff topology, then $M_\infty$ is $\delta-$hyperbolic.
\end{lem}

\begin{repprox}{pro:delta}
     Let $M$ a complete maximal $p-$submanifold in $\hyp^{p,q}$. If $M$ is $\delta-$hyperbolic, then there exists a constant $c=c(\delta)>0$ such that 
     \[\Scal_M<-c.\]
\end{repprox}
\begin{proof}
We divide the proof in two steps: first, we prove that if $M$ is $\delta-$hyperbolic, then $\Scal_M$ is uniformly bounded by some negative constant $-c$. Then, we prove that $c$ only depends on $\delta$.

Both proofs are by contradiction: we build a sequence of spacelike complete maximal $p-$submanifold which are $\delta-$hyperbolic, and we prove that such sequence converges (up to a subsequence) to a pseudo-flat, which is not $\delta-$hyperbolic. By Lemma~\ref{lem:deltaconv}, this is not possible, concluding the proof.

\step{Scalar curvature uniformly bounded}
We prove that if $\sup_M\Scal_M=0$, then $M$ is not $\delta-$hyperbolic.

Let $x_n$ a sequence in $M$ such that 
\[\lim_{n\to+\infty}\Scal_M(x_n)=\sup_M\Scal_M=0.\]

Let us fix a pointed spacelike totally geodesic $p-$submanifold $(x_0,H)$. For any $n\in\N$, let $\phi_n\in\mathrm{Isom}(\hyp^{p,q})$ an isometry such that 
\[\phi_n(x_n)=x_0,\qquad d_{x_n}\phi_n(T_{x_n}M)=H.\] 
The sequence $\phi_n(M)$ is then a sequence of spacelike complete maximal $p-$submanifold which are $\delta-$hyperbolic by hypothesis. Up to extracting a subsequence, by Proposition~\ref{pro:compact}, we have that \[M_\infty:=\lim_{n\to+\infty}\phi_n(M)\] is a complete maximal spacelike $p-$submanifold and $\Scal_{M_\infty}(x_0)=0$. By Theorem~\ref{pro:scal}, we have that $\Scal_{M_\infty}\equiv0$ and $M_\infty$ is a pseudo-flat $p-$submanifold, which is isometric to $\R^p$, hence it is not $\delta-$hyperbolic, leading to a contradiction.\cite{ONEILL}

\step{$c=c(\delta)$}
Let us fix $\delta>0$. By contradiction, let $M_n$ be a sequence of $\delta-$hyperbolic $p-$submanifolds in $\hyp^{p,q}$ such that 
\[\sup_{M_n}\Scal_{M_n}\to0.\]
We can build a sequence $x_n\in M_n$ such that
\[\Scal_{M_n}(x_n)>\sup_{M_n}\Scal_{M_n}-\frac{1}{n}.\]
For any $n\in\mathbb{N}$, let $\phi_n\in\mathrm{Isom}(\hyp^{p,q})$ an isometry such that 
\[\phi_n(x_n)=x_0,\qquad d_{x_n}\phi_n(T_{x_n}M_n)=H.\] 
As above, the limit \[M_\infty:=\lim_{n\to+\infty}\phi_n(M_n)\] is a copy of $\R^p$, leading to a contradiction and concluding the proof.
\end{proof}
\begin{rem}\label{rem:delta}
    The converse is false: indeed, all product submanifolds which are not pseudo-flat have strictly negative scalar curvature, and yet they are not $\delta-$hyperbolic, since they contains totally geodesic flat submanifolds. 
\end{rem}

Using Proposition~\ref{pro:compact}, we prove that complete spacelike maximal $p-$submanifolds with sufficiently negative scalar curvature are in fact uniformely negatively curved.
\begin{pro}\label{lem:delta}
    Let $p$ be an integer bigger than 1. There exists a positive constant $\varepsilon$ depending only on $p$ with the following property: let $M$ be a complete spacelike maximal $p-$submanifold of $\hyp^{p,q}$. If \[\sup_M\Scal_M<\varepsilon-p(p-1),\]
    then $M$ is uniformly negatively curved, that is, there exists a positive constant $c$ such that
    \[\sup_{\mathrm{Gr}_{2}(TM)}\sec_M\le-c<0.\]

    In particular, $M$ is $\delta-$hyperbolic, for some positive $\delta$.
\end{pro}
\begin{proof}
    By contradiction, let $M_n$ be a sequence of spacelike complete maximal $p-$submanifolds such that
    \begin{align*}
        &\sup_{M_n}\Scal_{M_n}<\frac{1}{n}-p(p-1),\\
        &\sup_{\mathrm{Gr}_{2}(TM_n)}\sec_{M_n}\ge0.
    \end{align*}
    We can build a sequence of points $x_n\in M_n$ such that \[\max_{\mathrm{Gr}_{2}(T_{x_n}M_n)}\sec_{M_n}>-1/n.\] 

    Let us fix a pointed spacelike totally geodesic $p-$submanifold $(x_0,H)$. For any $n\in\N$, let $\phi_n\in\mathrm{Isom}(\hyp^{p,q})$ an isometry such that 
\[\phi_n(x_n)=x_0,\qquad d_{x_n}\phi_n(T_{x_n}M_n)=H.\] 
By Proposition~\ref{pro:compact}, up to extract a subsequence, we can assume that \[M_\infty:=\lim_{n\to+\infty}\phi_n(M_n)\] is a complete spacelike maximal $p-$submanifold with
\begin{equation}\label{eq:secdelta}
    \max_{\mathrm{Gr}_{2}(T_{x_0}M_\infty)}\sec_{M_\infty}\ge0.
\end{equation}
However, $\Scal_{M_\infty}\equiv-p(p-1)$, hence $M_\infty$ is totally geodesic. This contradicts Equation~\eqref{eq:secdelta} and concludes the proof.
\end{proof}
\begin{rem}
    The converse holds if the supremum of the scalar curvature is sufficiently small.
\end{rem}

\printbibliography

\end{document}